\documentclass[11pt]{amsart}

\usepackage[T1]{fontenc}
\IfFileExists{lmodern.sty}{\usepackage{lmodern}%
  \IfFileExists{microtype.sty}{\usepackage{microtype}}{}}{}
\usepackage{amsmath,amssymb,mathtools}
\usepackage{enumitem}
\usepackage{booktabs,array}
\usepackage[colorlinks=true,linkcolor=blue,citecolor=blue,urlcolor=blue,
  pdftitle={The Conway--Parker algebra and the largest Fischer group},
  pdfauthor={Gerald H{\"o}hn}]{hyperref}
\allowdisplaybreaks
\numberwithin{equation}{section}

\newtheorem{theorem}{Theorem}[section]
\newtheorem{proposition}[theorem]{Proposition}
\newtheorem{lemma}[theorem]{Lemma}
\newtheorem{corollary}[theorem]{Corollary}
\theoremstyle{definition}
\newtheorem{definition}[theorem]{Definition}
\newtheorem{remark}[theorem]{Remark}
\newtheorem*{maintheorem}{Main Theorem}

\newcommand{\F}{\mathbb F}
\newcommand{\E}{\mathbb E}
\newcommand{\C}{\mathcal C}
\newcommand{\Ocal}{\mathcal O}
\newcommand{\Hcal}{\mathcal H}
\newcommand{\Rcal}{\mathcal R}
\newcommand{\Dcal}{\mathcal D}
\newcommand{\one}{\mathbf 1}
\newcommand{\Aut}{\operatorname{Aut}}
\newcommand{\St}{\operatorname{St}}
\newcommand{\Fi}{\mathrm{Fi}}
\newcommand{\eps}{\varepsilon}
\newcommand{\vth}{\vartheta}

\newcommand{\parity}{\operatorname{par}}
\newcommand{\ev}{\operatorname{ev}}

\title[The Conway--Parker algebra and the largest Fischer group]
{The Conway--Parker algebra\\
and the largest Fischer group}
\author{Gerald H\"ohn}
\address{Department of Mathematics, Kansas State University, Manhattan, Kansas 66506, USA}
\email{gerald@monstrous-moonshine.de}
\date{September 27, 2026}

\begin{document}

\begin{abstract}
We give a direct, self-contained construction of the three sporadic Fischer
groups $\Fi_{24}'$, $\Fi_{23}$, and $\Fi_{22}$ from the $783$-dimensional
Conway--Parker algebra.  We prove that its distinguished roots define
involutory algebra automorphisms whose projective actions generate the full
Fischer $3$-transposition group $\Fi_{24}$.  Its commutator subgroup gives
$\Fi_{24}'$, while $\Fi_{23}$ and $\Fi_{22}$ arise as centralizer quotients
associated with one and two commuting transpositions.

The root and frame geometry determines the group orders and leads to
elementary proofs of simplicity, as well as natural rank-three actions and
nonsplit central extensions.  The construction uses standard facts about
the Golay code, Parker's loop, and $M_{24}$.  It does not use the Monster
or previously known existence or order results for the Fischer groups.
Fischer's classification and later recognition theorems are used only
for the final identification.
\end{abstract}

\maketitle

\setcounter{tocdepth}{1}
\tableofcontents

\section*{Introduction}

Throughout, $\Fi_{24}$ denotes the full Fischer $3$-transposition group
and $\Fi_{24}'$ its simple subgroup of index two.  We construct the
largest group first from the $783$-dimensional Conway--Parker algebra;
the two smaller groups arise as centralizer quotients.  Until the final
identification, we write $G_{24}=G_{\rm CP}$, $G_{24}^+=G_{24}'$,
$G_{23}$, and $G_{22}$ for the constructed groups.

We start with the extended binary Golay code $\mathcal G_{24}\subseteq\F_2^I$,
Parker's code loop $P$, and the Mathieu group
$M_{24}=\Aut(\mathcal G_{24})$.  The $24$ coordinate positions give
commuting projective root reflections whose relations are the Golay
codewords.  They generate the cocode group $2^{12}$, and the set of these
$24$ reflections has stabilizer $2^{12}.M_{24}$.  The group orders and
simplicity are then proved from the algebra and its roots.
Fischer's classification and Aschbacher's local characterizations enter
only in Section~\ref{sec:identification}, where the constructed groups
are identified with the Fischer groups.

\paragraph{\textbf{The Conway--Parker algebra.}}
The $783$-dimensional representation and its interpretation through a
transposition algebra belong to Norton's early work on the largest Fischer
group; Griess lists Norton's \emph{Transposition algebras and the group $F_{24}$}
as forthcoming in his bibliography~\cite[reference~54]{GriessFriendly}.
Norton later published a structural account of the group~\cite{NortonFi24}.
The $24+759$ coordinate model used here was developed by Conway
and Parker from Parker's remarkable Moufang loop.  Its multiplication and
the three relevant vector shapes were recorded in the
\emph{Atlas}~\cite{Atlas}; Wilson gives a more extended account in
\cite[Sections~5.7.9--5.7.10]{WilsonBook}.  These sources describe the algebra concisely.  Our purpose is to supply
proofs from the Conway--Parker formulas that the root maps are algebra
automorphisms and that they generate the largest Fischer group.

Conway's 1985 paper
\emph{A simple construction for the Fischer--Griess monster group}
\cite{ConwayMonster} cites, as ``in press'' and with the date 1985, a paper
with R.~A. Parker entitled
\begin{quote}
\emph{A remarkable Moufang loop, with an application to the Fischer group
$\Fi_{24}$}.
\end{quote}
No published version of this paper is known; later bibliographies list it
as ``in preparation''~\cite{ConwayParker}.  Robert A. Wilson has informed
the author that he does not recall seeing a completed manuscript and
considers it possible that no finished paper was ever written.  The present
construction starts from the same explicit Conway--Parker data, but we do not
know how closely its proof follows the unpublished argument envisaged by
Conway and Parker.

\paragraph{\textbf{Earlier constructions.}}
Two earlier computer-free routes to $\Fi_{24}$ are relevant here.

The first goes through Fischer's theory of $3$-transposition groups.  The 1971 paper~\cite{FischerThree} is the published
first part of a substantially longer Warwick manuscript from 1969
\cite{FischerWarwick}.  The traditional existence construction for the
exceptional chain is inductive: $\Fi_{22}$ is constructed first, its local
geometry is used to construct $\Fi_{23}$, and the latter is then used to
construct $\Fi_{24}$.  The detailed published completion of the last step
was supplied by Marguerite-Marie Virotte-Ducharme.  Building on her 1985
thesis~\cite{VirotteThesis}, her memoir~\cite{VirotteDucharme} starts with
the Fischer graph of $\Fi_{23}$, constructs the required extension, proves
the point reflections to be automorphisms, establishes existence and
uniqueness in this extension framework, and determines the order.
Her construction is computer-free but takes $\Fi_{23}$ and its geometry
as input; here the order of construction is reversed.

The second route goes through the Monster.  Griess first constructs the
Monster as the automorphism group of his $196884$-dimensional algebra and
computes its order~\cite{GriessFriendly}.  For an element $z$ in the
Monster class now denoted $3A$, the relevant centralizer and normalizer
quotients are then identified with the simple and full largest Fischer
groups.  This is an ambient existence proof rather than a construction of
the $306936$-point Fischer geometry.  In Lemma~13.3 of~\cite{GriessFriendly}, Griess identifies the simple
centralizer quotient using the solution of the $O_2$-extraspecial problem.
For background on the classification of groups with large extraspecial
$2$-subgroups, see~\cite{SmithWidths,TimmesfeldExtraspecial}.
Aschbacher subsequently gave a complete published treatment of
Fischer's classification and uniqueness theorem
\cite[Part~I, Chapter~5]{AschbacherThree}, as well as the
involution-centralizer characterizations needed in the classification of
finite simple groups \cite[Part~II, Chapter~11]{AschbacherThree}.
The latter treatment also establishes existence using the Monster.

Starting from the Conway--Parker multiplication table, we prove that
the three root families are all the reflecting roots and that their root
maps are automorphisms, and we count the Fischer frames.  These results
give the orders and simplicity of the three Fischer groups.

\paragraph{\textbf{Outline.}}
In Section~\ref{sec:construction} we define the coordinate algebra and
the action of $H=2^{12}.M_{24}$.  Section~\ref{sec:roots} constructs the
roots.  We first prove antiunitarity of an octadic root map by its block
decomposition, then multiplicativity by the cubic-tensor identity proved
in Appendix~\ref{app:tensor-contractions}.  For arbitrary reflecting
roots, the root equations force the allowed pairings, and the
symmetric-square Gram matrix gives the bound $\binom{784}{2}=306936$.
The three constructed families attain the bound, proving completeness
and independence of the octad-pair choices in the duadic construction.

In Section~\ref{sec:three-transposition} we obtain the three-transposition
action and its semilinear lift.  Fixing one duadic root and propagating
signs through two octad graphs determines the pointwise stabilizer of the
standard frame; octadic switches supply the Mathieu translations.  Sylow conjugacy gives
conjugacy of frames, and counting frame--pentad incidences gives the
group order.

Section~\ref{sec:intrinsic-groups} treats the centralizer quotients, their
rank-three actions, simplicity, and central extensions, followed by the
identification with the Fischer groups.  In
Section~\ref{sec:integral}, quadratic reconstruction gives
the moment identities used to study the integral Eisenstein algebra
and its reduction modulo $1-\omega$.  Section~\ref{sec:integral} is
not used in Sections~\ref{sec:construction}--\ref{sec:intrinsic-groups}.
Appendix~\ref{app:factor-set} verifies the explicit Parker factor set.

\begin{maintheorem}
The Conway--Parker data determine a $783$-dimensional semilinear algebra
$A$ over $\E=\mathbb Q(\omega)$, a set $\Rcal$ of $306936$ norm-$9$ rays,
and a faithful permutation
group $G_{\rm CP}$ with the following properties.
\begin{enumerate}[label=\textup{\Alph*.},leftmargin=2.8em]
\item The group $H=\Aut_{\St}(P)=2^{12}.M_{24}$ acts by
semilinear algebra automorphisms; its $\E$-linear subgroup has shape
$2^{11}.M_{24}$.
\item The basic, octadic, and duadic roots satisfy $r*r=10r$.
Their root maps are antiunitary involutory algebra automorphisms, and
$\Rcal$ is precisely the set of rays of roots with this property.
In particular the duadic fibres are independent of their octad-pair
construction and every corresponding root reflection preserves $\Rcal$.
\item Their projective reflections form one generating
$3$-transposition class of size $306936$ in the finite centerless
group $G_{\rm CP}$.  Its index-two
commutator subgroup is perfect, and the standard $24$-frame has
stabilizer $2^{12}.M_{24}$.
\item All Fischer frames are conjugate, every commuting pentad lies in
exactly three frames, and
\[
 |G_{\rm CP}|=2^{22}3^{16}5^2 7^3\cdot11\cdot13\cdot17\cdot23\cdot29.
\]
\item The index-two subgroup $G_{24}^+$ and the explicitly defined
residues $G_{23}$ and $G_{22}$ are nonabelian simple.  Their orders are,
respectively,
\[
\begin{gathered}
2^{21}3^{16}5^2 7^3\cdot11\cdot13\cdot17\cdot23\cdot29,\\
2^{18}3^{13}5^2\cdot7\cdot11\cdot13\cdot17\cdot23,\qquad
2^{17}3^9 5^2\cdot7\cdot11\cdot13.
\end{gathered}
\]
The groups $G_{23}$ and $G_{22}$ have faithful primitive rank-three
actions on $31671$ and $3510$ points, respectively.  The centralizer of a
distinguished transposition in $G_{24}$ is isomorphic to
$C_2\times G_{23}$, and that of a distinguished transposition in $G_{23}$
is a perfect nonsplit double cover of $G_{22}$.  The $\E$-linear automorphism
group is the constructed perfect nonsplit central triple cover of
$G_{24}^+$; the full semilinear group is its split extension by a root
reflection, which inverts the scalar subgroup of order three.
The final recognition identifies
$(G_{24},G_{24}^+,G_{23},G_{22})$ with
$(\Fi_{24},\Fi_{24}',\Fi_{23},\Fi_{22})$.
\item The root span $L_{\rm CP}$ over $\mathbb Z[\omega]$ is an integral
Hermitian lattice of rank $783$, closed under $*$ and self-dual at every
odd prime.  Its discriminant module is $2$-primary and annihilated by $8$.
Reduction modulo $1-\omega$ gives a $783$-dimensional commutative
$\F_3$-algebra with a nondegenerate symmetric invariant form, a faithful
$G_{24}$-action, and $306936$ distinct isotropic idempotents.
\end{enumerate}
\end{maintheorem}

\paragraph{\textbf{Relation with vertex operator algebras.}}
In forthcoming joint work with Ching Hung Lam and Hiroshi Yamauchi, we will
give a substantially shorter vertex-operator-algebraic construction of the
$783$-dimensional algebra.  That construction explains the axes and their
Miyamoto involutions conceptually.  The two papers are independent.  The aim here is to obtain the groups,
their orders, and their simplicity from the coordinate algebra by frame
geometry and normal-subgroup arguments.

\paragraph{\textbf{Computational checks.}}
Computational checks were used for identities and finite tables during
the development of the paper; the proofs below do not depend on these
checks.

The characteristic-zero constructions of $\Fi_{24}'$, $\Fi_{23}$, and
$\Fi_{22}$, including their order and simplicity arguments, have been
formalized in Lean and verified by its kernel from an earlier version of
this manuscript as part of the \textsc{Atlas} project~\cite{AtlasLean}.
The integral Eisenstein and modular-reduction results below are not
included in that formalization.

\paragraph{\textbf{Acknowledgments.}}
The author thanks Robert A. Wilson for useful historical information about
the Conway--Parker construction, in particular his recollections concerning
the announced but apparently unpublished paper.
The author also thanks Ching Hung Lam and Hiroshi Yamauchi for discussions,
and Caleb Fernelius for his interest and participation in the Spring 2026
lectures on finite simple groups.

\section*{Global notation}
Only symbols used across several sections are collected here; auxiliary
notation local to a proof is introduced where it is needed.

{\small
\begin{description}[style=multiline,leftmargin=4.4cm,labelwidth=3.9cm,
  itemsep=.25ex,parsep=0pt,topsep=.6ex]
\item[$\F_2$] the binary field.
\item[$I$] the set of $24$ coordinate positions.
\item[$\C=\mathcal G_{24}$] the extended binary Golay code in $\F_2^I$.
\item[$\one$] the all-one word of $\C$.
\item[$\Ocal$] the set of the $759$ Golay octads.
\item[$\lambda_s$] the number of Golay octads through a fixed $s$-subset,
  for $0\leq s\leq5$.
\item[$\C^*=\F_2^I/\C$] the Golay cocode.
\item[$\langle c,\delta\rangle$] the canonical code--cocode pairing
  $\C\times\C^*\to\F_2$.
\item[$P$, $d\mapsto\bar d$] Parker's code loop and its quotient map
  $P\to\C$.
\item[$\Omega$] the chosen central lift in $P$ of the all-one word $\one$.
\item[$H=\Aut_{\St}(P)$] the standard Parker-loop automorphism group,
  $H\cong2^{12}.M_{24}$.
\item[$H^+$] the index-two subgroup of $H$ fixing $\Omega$.
\item[$T_\delta$] the semilinear cocode automorphism associated with
  $\delta\in\C^*$.
\item[$\E=\mathbb Q(\omega)$] the coefficient field
  $\mathbb Q(\sqrt{-3})$.
\item[$\mu_3=\{1,\omega,\omega^2\}$] the group of cube roots of unity,
  with $\omega\ne1$.
\item[$\vth=\omega-\bar\omega$] the distinguished square root
  $\sqrt{-3}$.
\item[$U$, $W$] the $24$- and $759$-dimensional coordinate summands.
\item[$A=A_\E=U\oplus W$] the $783$-dimensional Conway--Parker algebra
  over $\E$.
\item[$A_{\mathbb C}$] its complex scalar extension.
\item[$*$] the commutative conjugate-bilinear product on $A$.
\item[$(x,y)$] the Hermitian scalar product on $A$, linear in the first
  variable and conjugate-linear in the second.
\item[$\Phi(x,y,z)$] the symmetric trilinear form $(x,y*z)$.
\item[$B_O$, $B_O^*$] the Golay code shortened off an octad $O$ and its
  multiplicative character group.
\item[$\C_p$, $\C_p^*$] the Golay code shortened off a duad $p$ and its
  multiplicative character group.
\item[$r_i$, $r_{O,\chi}$, $r_{p,\psi}$] the basic, octadic, and duadic
  roots.
\item[$\tau_r$] the conjugate-linear root map
  $\tau_r(x)=x*r-(r,x)r$.
\item[{\([r]=\mu_3r\)}] the three-element root ray of $r$.
\item[$\Rcal_{\rm b}$, $\Rcal_{\rm o}$, $\Rcal_{\rm d}$] the basic,
  octadic, and duadic root-ray families.
\item[$\Rcal$] the full set of $306936$ root rays.
\item[$L_{\rm CP}$] the integral Eisenstein root lattice
  $\mathbb Z[\omega]\langle r:[r]\in\Rcal\rangle$.
\item[$Z=\mu_3\operatorname{id}_A$] the scalar subgroup of order $3$.
\item[$d_{[r]}$] the ray permutation induced by $\tau_rZ$.
\item[$\Dcal_{\rm CP}$] the class of all projective root reflections.
\item[$G_{\rm CP}$] the faithful permutation group
  $\langle\Dcal_{\rm CP}\rangle\leq\operatorname{Sym}(\Rcal)$.
\item[$\mathcal F_0$, $E_0$] the standard Golay frame
  $\{d_i:i\in I\}$ and the elementary abelian group it generates.
\item[$\Gamma=\Aut_{\rm sl}(A,*)$] all $\E$-linear and
  $\E$-conjugate-linear algebra automorphisms; $\overline\Gamma=\Gamma/Z$.
\item[$\widehat G$, $\widehat G^+$] the semilinear root-reflection group
  and its perfect $\E$-linear subgroup.
\item[$G_{24}=G_{\rm CP}$, $G_{24}^+$] the constructed full group and
  its index-two commutator subgroup.
\item[$E_S$, $C_S$, $D_S$] the group generated by a basic subset $S$,
  its pointwise centralizer, and its set of remaining commuting
  distinguished involutions, as in \eqref{eq:residue-definitions}.
\item[$G_{23}$, $G_{22}$] the one- and two-point residue quotients
  defined in \eqref{eq:intrinsic-residue-groups}.
\item[$\Fi_{24}$, $\Fi_{24}'$, $\Fi_{23}$, $\Fi_{22}$] the conventional
  Fischer names identified with the constructed groups in
  Section~\ref{sec:identification}.
\end{description}
}

\section{The Conway--Parker algebra and its base group}\label{sec:construction}

\subsection{The Golay code, its cocode, and Parker's loop}

Let $I$ be a set of cardinality $24$, let
\[
   \C=\mathcal G_{24}\subseteq \F_2^I
\]
be the extended binary Golay code, and let $\one\in\C$ be the all-one word.
Put
\[
              \Ocal:=\{O\in\C:|O|=8\}
\]
for the set of Golay octads.  We identify a word with its support when convenient.  Since $\C$ is
self-dual, the dot product induces a perfect pairing
\begin{equation}
   \langle\ ,\ \rangle:\C\times \C^*\longrightarrow \F_2,
   \qquad
   \C^*:=\F_2^I/\C . \label{eq:1.1}
\end{equation}
Thus $\C^*$ is the Golay cocode.  If $\delta\in\C^*$ and $\widehat\delta$
is any representative in $\F_2^I$, then
\[
   \langle c,\delta\rangle=c\cdot\widehat\delta
\]
is independent of the representative.  The parity
\begin{equation}
   \parity(\delta):=\langle\one,\delta\rangle
              =|\widehat\delta|\pmod2 \label{eq:1.2}
\end{equation}
is also well defined, because every Golay word has even weight.

\begin{lemma}\label{lem:Witt-numerology}
The octads form the Steiner system $S(5,8,24)$.  If $S\subseteq I$ has
size $s\leq5$, then the number of octads through $S$ is
\begin{equation}
 \lambda_s=\frac{\binom{24-s}{5-s}}{\binom{8-s}{5-s}},
 \qquad
 (\lambda_0,\ldots,\lambda_5)=(759,253,77,21,5,1).          \label{eq:Witt-lambda}
\end{equation}
For a fixed octad $O$, the octad intersection distribution is
\begin{equation}
\begin{array}{c|rrrr}
 |D\cap O|&0&2&4&8\\ \hline
 \#D&30&448&280&1
\end{array}                                                     \label{eq:Witt-distributions}
\end{equation}
For a fixed duad $p$, the numbers of octads meeting $p$ in $2$, $1$, $0$
points are respectively
\begin{equation}
                              77,\qquad352,\qquad330.          \label{eq:Witt-duad}
\end{equation}
If $p\subset O$, then among the octads $D\ne O$ through $p$, exactly
$16$ meet $O$ in $p$ and exactly $60$ meet $O$ in four points.  Finally,
if $i\notin O$, the numbers of octads through $i$ meeting $O$ in
$0$, $2$, $4$ points are
\begin{equation}
                              15,\qquad168,\qquad70,            \label{eq:Witt-point-refinement}
\end{equation}
whereas, for $i\in O$, the numbers of octads $D\ne O$ through $i$
meeting $O$ in $2$, $4$ points are
\begin{equation}
                              112,\qquad140.                    \label{eq:Witt-point-inside}
\end{equation}
\end{lemma}

\begin{proof}
Formula~\eqref{eq:Witt-lambda} is the standard parameter formula for the
Witt design.  For a fixed octad, double-counting incidences with its points
and pairs, together with the possible intersection sizes $0$, $2$, $4$, $8$, gives
the table in \eqref{eq:Witt-distributions}.  Equation~\eqref{eq:Witt-duad} follows
from $\lambda_2=77$ and
\[
 2(\lambda_1-\lambda_2)=352,
 \qquad 759-77-352=330.
\]
If $p\subset O$
and $a_j$ counts the other octads through $p$ meeting $O$ in $j$ points,
then
\[
 a_2+a_4=\lambda_2-1=76,
 \qquad 2a_4=6(\lambda_3-1)=120,
\]
which gives $a_2=16$ and $a_4=60$.

For $i\notin O$, let $b_j$ count octads through $i$ meeting $O$ in $j$
points.  Counting incidences with points and pairs of $O$ gives
\[
 b_0+b_2+b_4=\lambda_1,
 \quad 2b_2+4b_4=8\lambda_2,
 \quad b_2+6b_4=\binom82\lambda_3,
\]
which yields \eqref{eq:Witt-point-refinement}.  For $i\in O$, the
corresponding counts among $D\ne O$ satisfy
\[
 n_2+n_4=\lambda_1-1=252,\qquad
 n_2+3n_4=7(\lambda_2-1)=532,
\]
giving \eqref{eq:Witt-point-inside}.

The octads span $\C$.  Indeed, choose five points of a dodecad $C$
and let $D$ be their unique octad.  The intersection $C\cap D$ has even
size at least five, and cannot have size eight, since $C+D$ would then
have weight four.  Thus $|C\cap D|=6$, and $C=D+(C+D)$ is a sum of
two octads.
For a weight-$16$ word $c=\one+O$, choose an octad $D$ disjoint
from $O$; there are $30$ choices by \eqref{eq:Witt-distributions}.
Then $c=D+(c+D)$ is a sum of two octads, and adding $O$ expresses
$\one$ as a sum of three.  These are all nonzero Golay weights.
\end{proof}

Let $P$ be Parker's loop.  We write
\[
    P\longrightarrow \C,\qquad d\longmapsto \bar d,
\]
for its quotient by $\langle-1\rangle$, and choose a lift
$\Omega\in P$ of $\one$.  We use the standard code-loop relations
\begin{align}
 d^2&=(-1)^{|\bar d|/4}, \label{eq:1.3}\\
 de&=(-1)^{|\bar d\cap\bar e|/2}ed, \label{eq:1.4}\\
 (de)f&=(-1)^{|\bar d\cap\bar e\cap\bar f|}d(ef). \label{eq:1.5}
\end{align}
These relations determine the square, commutator, and associator maps of
$P$; see Griess~\cite{GriessCodeLoops}.

\begin{lemma}\label{lem:Omega-central}
The element $\Omega$ is a central involution.
\end{lemma}

\begin{proof}
Equation \eqref{eq:1.3} gives
$\Omega^2=(-1)^{24/4}=1$.  If $d\in P$, then \eqref{eq:1.4} gives
\[
 [\Omega,d]=(-1)^{|\bar d|/2}=1,
\]
since every Golay word has weight divisible by $4$.  Finally, by
\eqref{eq:1.5},
\[
 [\Omega,d,e]=(-1)^{|\bar d\cap\bar e|}=1,
\]
because $\C=\C^\perp$.  Hence $\Omega\in Z(P)$.
\end{proof}

Let
\[
       H:=\Aut_{\St}(P)
\]
be the group of standard automorphisms of $P$, namely those automorphisms
which induce coordinate permutations preserving $\C$.  There is an exact
sequence
\begin{equation}
  1\longrightarrow \C^*\longrightarrow H
   \stackrel{\pi}{\longrightarrow}M_{24}\longrightarrow1, \label{eq:1.6}
\end{equation}
and this extension is nonsplit.  In Atlas notation,
\begin{equation}
                         H\cong 2^{12}.M_{24}. \label{eq:1.7}
\end{equation}
The kernel in \eqref{eq:1.6} is the cocode.  For
$\delta\in\C^*$ define
\begin{equation}
      \sigma_\delta(d)
      =(-1)^{\langle\bar d,\delta\rangle}d. \label{eq:1.8}
\end{equation}
Then $\sigma_\delta\in\Aut(P)$, and
$\delta\mapsto\sigma_\delta$ identifies $\C^*$ with the kernel of $\pi$.
Indeed, an automorphism inducing the identity on $\C$ has the form
$d\mapsto\chi(\bar d)d$, and preservation of multiplication says precisely
that $\chi:\C\to\{\pm1\}$ is a character.  The perfect pairing
\eqref{eq:1.1} identifies the character group of $\C$ with $\C^*$.
Throughout this identification is written multiplicatively: an additive
functional $\lambda$ corresponds to the character
$c\mapsto(-1)^{\lambda(c)}$.  Accordingly, for a binary subspace
$V\leq\C$, the notation $V^*$ below means
$\operatorname{Hom}(V,\{\pm1\})$ in this multiplicative convention.
The existence of lifts of all elements of $M_{24}$ and the nonsplitting of
\eqref{eq:1.6} are standard facts about the Parker loop; see
\cite{GriessCodeLoops,SeysenMonster}.

\begin{lemma}[Parker factor set]\label{lem:explicit-factor-set}
Let $\C$ be a doubly even self-orthogonal binary code with ordered basis
$g_1$, $\dots$, $g_n$.  For
$a=\sum a_i g_i$ put
\[
 \iota(a,b,c)=|a\cap b\cap c|\pmod2
\]
and define the trilinear form
\[
 \Theta(a,b,c)=\sum_i\ \sum_{j<k}a_i b_j c_k\,
                 \iota(g_i,g_j,g_k).
\]
Let $\beta$ be the bilinear form specified on basis vectors by
\[
\begin{aligned}
 \beta(g_i,g_i)&=\frac{|g_i|}{4},\\
 \beta(g_i,g_j)&=\frac{|g_i\cap g_j|}{2} &&(i<j),\\
 \beta(g_i,g_j)&=0 &&(i>j).
\end{aligned}
\]
Set
\begin{equation}
                f(a,b)=\Theta(a,b,b)+\beta(a,b)\pmod2.              \label{eq:1.12}
\end{equation}
Then $P=\C\times\{\pm1\}$ with
\[
 (a,\zeta_1)(b,\zeta_2)
   =(a+b,\zeta_1\zeta_2(-1)^{f(a,b)})
\]
satisfies the square, commutator, and associator identities
\eqref{eq:1.3}--\eqref{eq:1.5}.  Hence Parker's loop may be constructed
explicitly from the Golay code and a choice of ordered basis.
\end{lemma}

The verification from the ordered-basis formula is given in Appendix~\ref{app:factor-set}.

For $h\in H$, define $\eps(h)\in\F_2$ by
\begin{equation}
                \Omega^h=(-1)^{\eps(h)}\Omega. \label{eq:1.9}
\end{equation}
This is well defined because $h$ fixes $-1$ and induces an element of
$M_{24}$, which fixes the all-one word.  Applying successive
automorphisms to \eqref{eq:1.9}, and using that they fix $-1$, shows
that $\eps:H\to\F_2$ is a homomorphism.  On the cocode subgroup,
\begin{equation}
       \eps(\sigma_\delta)=\langle\one,\delta\rangle
                           =\parity(\delta). \label{eq:1.10}
\end{equation}
Consequently
\begin{equation}
       H^+:=\ker\eps\cong 2^{11}.M_{24} \label{eq:1.11}
\end{equation}
is the index-two subgroup fixing $\Omega$.  We use the extension itself,
without choosing a complement to the cocode.

\subsection{The algebra and the complementary-octad phase}

Let
\[
      \E=\mathbb Q(\omega)=\mathbb Q(\sqrt{-3}),
      \qquad \mu_3:=\{1,\omega,\omega^2\},\quad \omega^3=1,\quad \omega\ne1,
\]
and put
\begin{equation}
      \vth:=\omega-\bar\omega=\sqrt{-3}. \label{eq:2.1}
\end{equation}
Thus $\bar\vth=-\vth$.
Let $U$ have basis $u_i$ ($i\in I$), and let $W$ be the $\E$-space
generated by symbols $x_d$, where $d\in P$ and $\bar d$ is an octad.
Our $u_i$, $x_d$ are Wilson's $a_i$, $e_d$ (his $A_i=8a_i$); the
normalization is unchanged.  We retain the
present letters to distinguish coordinate vectors from roots, characters,
and the algebra $A$.  The signed octad vectors are subject to
\begin{equation}
                           x_{-d}=-x_d. \label{eq:2.2}
\end{equation}
There are two lifts of each of the $759$ octads, so $\dim W=759$.  Put
\begin{equation}
                            A:=U\oplus W. \label{eq:2.3}
\end{equation}
Thus $\dim_\E A=24+759=783$.  We also write $A_\E=A$ and put
\[
 A_{\mathbb C}=\mathbb C\otimes_\E A_\E.
\]
The conjugate-bilinear product below and the Hermitian form of
\eqref{eq:4.1} extend by the same coordinate formulas.  Complex
orthonormal-basis arguments take place in $A_{\mathbb C}$.  Unless
explicitly extended, algebra automorphisms act
on $A=A_\E$, with scalar action either the identity or conjugation;
we denote the group of these semilinear algebra automorphisms by
$\Aut_{\rm sl}(A,*)$.  Their extensions to $A_{\mathbb C}$ are respectively
complex-linear or conjugate-linear.

The product $*$ is commutative and conjugate-bilinear:
\begin{equation}
  (\alpha x+\beta y)*z
     =\bar\alpha(x*z)+\bar\beta(y*z), \label{eq:2.4}
\end{equation}
and hence likewise in the second argument.  On the coordinate basis it is
defined as follows:
\begin{align}
 128u_i*u_i
    &=-81u_i+15\sum_{j\ne i}u_j, \label{eq:2.5}\\
 128u_i*u_j
    &=15u_i+15u_j-\sum_{k\notin\{i,j\}}u_k
      \qquad(i\ne j), \label{eq:2.6}\\
 16u_i*x_d
    &=\begin{cases}
        3x_d,&i\in\bar d,\\
        -x_d,&i\notin\bar d,
      \end{cases} \label{eq:2.7}\\
 2x_d*x_d
    &=3\sum_{i\in\bar d}u_i-\sum_{i\notin\bar d}u_i. \label{eq:2.8}
\end{align}
If $d\ne\pm f$, then two distinct octads meet in $0$, $2$, or $4$ points,
and we put
\begin{equation}
  2x_d*x_f=
  \begin{cases}
     x_{df},&|\bar d\cap\bar f|=4,\\[2mm]
     \vth\,x_{df\Omega},&|\bar d\cap\bar f|=0,\\[2mm]
     0,&|\bar d\cap\bar f|=2.
  \end{cases} \label{eq:2.9}
\end{equation}
For intersection $4$, the support of $df$ is the octad
$\bar d+\bar f$.  For intersection $0$, the support of $df\Omega$ is the
octad $\bar d+\bar f+\one$, the complement of
$\bar d\cup\bar f$.

\begin{lemma}\label{lem:well-defined}
The multiplication table \eqref{eq:2.5}--\eqref{eq:2.9} is well defined on the
quotient \eqref{eq:2.2}, and it is commutative.
\end{lemma}

\begin{proof}
Replacing $d$ by $-d$ changes both sides of every off-diagonal formula by a
minus sign, while the diagonal formula is unchanged.  If two octads meet in
$0$ or $4$ points, then \eqref{eq:1.4} gives $df=fd$; if they meet in $2$
points, the product is defined to be zero.  Finally, $\Omega$ is central by
Lemma~\ref{lem:Omega-central}.
\end{proof}

\begin{lemma}\label{lem:scalar-automorphisms}
For every $\lambda\in\mu_3$, scalar multiplication by $\lambda$ is an
$\E$-linear algebra automorphism of $A$.
\end{lemma}

\begin{proof}
Conjugate-bilinearity and $\bar\lambda^2=\lambda$ give
\[
       (\lambda x)*(\lambda y)=\bar\lambda^2(x*y)
                              =\lambda(x*y).
\]
\end{proof}

Remark~\ref{rem:Wilson-phase} derives the phase and scale of the
complementary-octad coefficient $\vth$ from the root calculation.

\subsection{The nonsplit frame group}

For $h\in H$, let $\pi_h\in M_{24}$ be the induced coordinate permutation.
Define a semilinear right action of $H$ on $A$ by
\begin{equation}
      u_i^h=u_{\pi_h(i)},
      \qquad
      x_d^h=x_{d^h}, \label{eq:3.1}
\end{equation}
and extend scalars according to
\begin{equation}
   (\lambda x)^h=
   \begin{cases}
      \lambda x^h,&\eps(h)=0,\\
      \bar\lambda x^h,&\eps(h)=1.
   \end{cases} \label{eq:3.2}
\end{equation}
This is well defined because $h$ fixes $-1$ and hence
$x_{(-d)^h}=x_{-d^h}=-x_{d^h}$.  Since $\eps$ is a homomorphism and
$(d^h)^k=d^{hk}$, equations \eqref{eq:3.1}--\eqref{eq:3.2} define a
semilinear right action of $H$ on $A$.  We use ordinary function
composition in automorphism groups; the corresponding embedding therefore
sends $h$ to $x\mapsto x^{h^{-1}}$.

\begin{theorem}[Frame-group action]\label{thm:base-action}
The maps \eqref{eq:3.1}--\eqref{eq:3.2} give a faithful action
\begin{equation}
       2^{12}.M_{24}=\Aut_{\St}(P)
          \hookrightarrow \Aut_{\rm sl}(A,*). \label{eq:3.3}
\end{equation}
Every $h\in H^+\cong2^{11}.M_{24}$ acts $\E$-linearly, and every element of
the other coset acts conjugate-linearly.
\end{theorem}

\begin{proof}
Because the product is conjugate-bilinear, it is enough to check
\begin{equation}
                         (a*b)^h=a^h*b^h \label{eq:3.4}
\end{equation}
on the coordinate basis.

The products in $U*U$ are preserved because their coefficients are real and
the formulas depend only on whether the two coordinate labels are equal.
The mixed products are preserved because
\begin{equation}
             i\in\bar d
       \quad\Longleftrightarrow\quad
             \pi_h(i)\in\overline{d^h}. \label{eq:3.5}
\end{equation}
The diagonal products $x_d*x_d$ are preserved for the same reason; any sign
in $d^h$ occurs twice and disappears.

Let $d\ne\pm f$.  If $|\bar d\cap\bar f|=2$, both sides of \eqref{eq:3.4} are
zero.  If $|\bar d\cap\bar f|=4$, then
\begin{equation}
  (2x_d*x_f)^h=x_{(df)^h}=x_{d^hf^h}=2x_{d^h}*x_{f^h}. \label{eq:3.6}
\end{equation}
It remains to consider disjoint octads.  Write $\eps_h=\eps(h)$.  Since
$\bar\vth=-\vth$ and $\Omega^h=(-1)^{\eps_h}\Omega$, we have
\begin{align*}
 (2x_d*x_f)^h
   &=(\vth x_{df\Omega})^h\\
   &=(-1)^{\eps_h}\vth\,x_{(df\Omega)^h}\\
   &=(-1)^{\eps_h}\vth\,
       x_{d^hf^h(-1)^{\eps_h}\Omega}\\
   &=\vth x_{d^hf^h\Omega}\\
   &=2x_{d^h}*x_{f^h}.
\end{align*}
The sign from conjugating $\vth$ is cancelled exactly by the sign with which
$h$ moves the chosen lift $\Omega$.

Thus every $h\in H$ preserves the multiplication.  To prove faithfulness,
suppose $h$ acts trivially.  Its action on $U$ shows that $\pi_h=1$.  If
$\eps(h)=1$, then $h$ sends $\vth u_i$ to $-\vth u_i$, a contradiction; hence
$\eps(h)=0$.  Thus $h=\sigma_\delta$ for some $\delta\in\C^*$.  Triviality on
all octad vectors gives
\begin{equation}
                  \langle D,\delta\rangle=0
       \qquad\text{for every octad }D. \label{eq:3.7}
\end{equation}
The octads span the Golay code, so the perfect pairing \eqref{eq:1.1} gives
$\delta=0$.  Hence $h=1$.
\end{proof}

The cocode subgroup acts as follows.

\begin{corollary}\label{cor:cocode}
For $\delta\in\C^*$ put
\begin{equation}
 T_\delta(u_i)=u_i,
 \qquad
 T_\delta(x_d)=(-1)^{\langle\bar d,\delta\rangle}x_d, \label{eq:3.8}
\end{equation}
and extend $T_\delta$ linearly if $\parity(\delta)=0$ and
conjugate-linearly if $\parity(\delta)=1$.  Then $T_\delta$ is an algebra
automorphism.  In the complementary-octad case the entire verification is
the identity
\begin{equation}
 \begin{split}
 \parity(\delta)+\langle D+F+\one,\delta\rangle
  &=\langle\one,\delta\rangle
    +\langle D+F+\one,\delta\rangle\\
  &=\langle D+F,\delta\rangle.
 \end{split} \label{eq:3.9}
\end{equation}
\end{corollary}

\begin{proof}
This is Theorem~\ref{thm:base-action} for the kernel of \eqref{eq:1.6}.
Equation \eqref{eq:3.9} says precisely that the sign from conjugating $\vth$
and the sign attached to the complementary octad combine to the product of
the two input signs.
\end{proof}

\section{Root equations and the ray configuration}\label{sec:roots}

\subsection{The Hermitian cubic and the basic roots}

Equip $A$ with the Hermitian scalar product, linear in the first variable and
conjugate-linear in the second, determined on the signed Parker
generators by
\begin{equation}
 (u_i,u_j)=\frac18\delta_{ij},
 \qquad U\perp W,
 \qquad
 (x_d,x_f)=
 \begin{cases}
  1,&f=d,\\
 -1,&f=-d,\\
  0,&\bar f\ne\bar d.
 \end{cases}                                                    \label{eq:4.1}
\end{equation}
When a calculation is independent of the choice of sign on an octad line,
$x_D$ denotes any chosen unit generator of the line above $D\in\Ocal$.

\begin{proposition}\label{prop:cubic}
The Hermitian form is preserved semilinearly: even elements of $H$ are
unitary, whereas odd elements satisfy $(x^h,y^h)=\overline{(x,y)}$ and are
antiunitary.  The cubic
\[
             \Phi(x,y,z)=(x,y*z)
\]
is symmetric and $\E$-trilinear on $A$, and complex-trilinear on
$A_{\mathbb C}$ after scalar extension.
\end{proposition}

\begin{proof}
The first assertion follows immediately from the coordinate action in
\eqref{eq:3.1}--\eqref{eq:3.2}.  Conjugate-bilinearity of $*$ and
conjugate-linearity of the second argument of the scalar product make
$\Phi$ $\E$-linear in all three variables.  Symmetry is checked on
the coordinate basis from \eqref{eq:2.5}--\eqref{eq:2.9}; the Parker signs in
the octad products are precisely the same signs under every permutation of
the three entries.
\end{proof}
Following the terminology of the \emph{Atlas}, we call a vector
$r\in A$ a \emph{root} if
\begin{equation}
                         (r,r)=9,\qquad r*r=10r.                 \label{eq:root-equations}
\end{equation}
Here ``root'' is Atlas terminology, not Lie-theoretic; the associated
idempotent is $r/10$.

For $r\in A$, define the conjugate-linear map
\begin{equation}
                  \tau_r(x)=x*r-(r,x) r. \label{eq:4.2}
\end{equation}

For $\lambda\in\mu_3$, direct substitution gives
\begin{equation}
                         \tau_{\lambda r}=\lambda^2\tau_r.      \label{eq:root-scalar}
\end{equation}
Thus the $\mu_3$-orbit of $r$ determines a projective root map.
For a root $r$, put $[r]=\mu_3r$; the word \emph{ray} denotes this
three-element set, not its entire complex line.

\begin{lemma}[Symmetric root maps]\label{lem:symmetric-root-maps}
For every $r\in A$,
\[
                       (\tau_r x,y)=(\tau_r y,x).
\]
Consequently $\tau_r$ is antiunitary if and only if $\tau_r^2=1$.
\end{lemma}
\begin{proof}
The expression
\[
       (\tau_r x,y)=\overline{\Phi(x,y,r)}-(r,x)(r,y)
\]
is symmetric in $x$, $y$.  It follows that
$(\tau_r^2x,y)=(\tau_r y,\tau_r x)$; nondegeneracy of the Hermitian
form proves the equivalence.
\end{proof}

\begin{lemma}\label{lem:root-covariance}
Let $g$ be an $\E$-linear unitary or $\E$-conjugate-linear antiunitary
algebra automorphism of $A$.  Then, for every $r\in A$,
\begin{equation}
                         \tau_{g(r)}=g\tau_rg^{-1}.             \label{eq:root-covariance}
\end{equation}
\end{lemma}

\begin{proof}
For $x\in A$, multiplicativity gives
\[
 g\bigl(g^{-1}x*r\bigr)=x*g(r).
\]
If $g$ is linear, then
$(r,g^{-1}x)=(g(r),x)$ and scalars pass
through $g$ unchanged.  If $g$ is conjugate-linear, antiunitarity gives
\[
 (g(r),x)
   =\overline{(r,g^{-1}x)},
\]
and conjugate-linearity applies the same conjugation to the scalar
multiplying $r$.  Substitution in \eqref{eq:4.2} proves
\eqref{eq:root-covariance} in both cases.
\end{proof}

\begin{lemma}\label{lem:nonorthogonal-rigidity}
Let $r$, $s$ have norm $9$ and satisfy $r*r=10r$, $s*s=10s$.  Assume that
$\tau_r$ is antiunitary.  If
$c=(r,s)\in\mu_3$, then
\begin{equation}
                         \tau_r(s)=\bar c\,s.                  \label{eq:nonorthogonal-rigidity}
\end{equation}
No algebra-automorphism property of $\tau_r$ is required.
\end{lemma}

\begin{proof}
By Proposition~\ref{prop:cubic}, cubic symmetry and the root equation for
$s$ give
\[
 (s,\tau_r(s))
   =(s,s*r)-(s,cr)
   =10c-\bar c^{\,2}=9c.
\]
Both vectors have norm $9$, so equality holds in Cauchy--Schwarz.  Hence
$\tau_r(s)=\lambda s$ for some $|\lambda|=1$; the displayed inner product
gives $\bar\lambda=c$, proving \eqref{eq:nonorthogonal-rigidity}.
\end{proof}
For $i\in I$, let
\begin{equation}
                  r_i=-7u_i+\sum_{j\ne i}u_j. \label{eq:4.3}
\end{equation}
Let $t_i=T_{\{i\}+\C}$ be the cocode transformation associated with the
singleton $\{i\}$.

\begin{proposition}\label{prop:basic-axis}
For every $i\in I$,
\begin{equation}
                         \tau_{r_i}=t_i. \label{eq:4.4}
\end{equation}
In particular, $\tau_{r_i}$ is a conjugate-linear antiunitary
involutory algebra automorphism.  Moreover,
\begin{equation}
     (r_i,r_i)=9,
     \qquad r_i*r_i=10r_i, \label{eq:4.5}
\end{equation}
so $r_i/10$ is an idempotent.
\end{proposition}

\begin{proof}
A direct use of \eqref{eq:2.5}--\eqref{eq:2.7} gives
\begin{equation}
  u_j*r_i=u_j+(r_i,u_j) r_i, \label{eq:4.6}
\end{equation}
and
\begin{equation}
  x_d*r_i=
  \begin{cases}
     -x_d,&i\in\bar d,\\
     x_d,&i\notin\bar d.
  \end{cases} \label{eq:4.7}
\end{equation}
Since $(r_i,x_d)=0$, equations
\eqref{eq:4.2}, \eqref{eq:4.6}, and \eqref{eq:4.7} prove \eqref{eq:4.4}.  The norm in
\eqref{eq:4.5} is
\[
        \frac18(49+23)=9.
\]
Finally $t_i(r_i)=r_i$, so applying \eqref{eq:4.2} to $r_i$ gives
$r_i*r_i-9r_i=r_i$, and hence $r_i*r_i=10r_i$.
\end{proof}

\begin{corollary}\label{cor:relations}
The $24$ basic involutions commute and generate the cocode:
\begin{equation}
       \langle t_i:i\in I\rangle\cong\C^*\cong2^{12}. \label{eq:4.8}
\end{equation}
More precisely, for $S\subseteq I$,
\begin{equation}
                    \prod_{i\in S}t_i=1
       \quad\Longleftrightarrow\quad S\in\C. \label{eq:4.9}
\end{equation}
\end{corollary}

\begin{proof}
The product in \eqref{eq:4.9} is the cocode transformation associated with
$S+\C$.  It is trivial precisely when this cocode element is zero.
\end{proof}

\subsection{Octadic roots and their block decomposition}

Fix an octad $O\in\Ocal$, choose a lift $o\in P$, and put
\[
                     X:=O^c=I\setminus O=\one+O.
\]
Define the shortened code
\[
       B_O:=\{c\in\C:c\cap O=\varnothing\}.
\]
Write
\[
               B_O^*:=\operatorname{Hom}(B_O,\{\pm1\})
\]
for its multiplicative character group.  For $j\in X$ write
\begin{equation}
                    \ev_j(b):=(-1)^{[j\in b]}\qquad(b\in B_O)       \label{eq:evaluation-character}
\end{equation}
for the corresponding evaluation character.  The code $B_O$ is the
$[16,5,8]$ code $RM(1,4)$ on $X$.  Its
nonzero words are the thirty affine hyperplanes of $X\cong\F_2^4$ and the
word $X$ itself.

\begin{lemma}\label{lem:elem-abelian}
Let $K\le\C$ be a subcode with
\begin{equation}
 |c|\equiv0\ (8),\qquad |c\cap c'|\equiv0\ (4),\qquad
 |c\cap c'\cap c''|\equiv0\ (2)                                 \label{eq:5.0}
\end{equation}
for all $c$, $c'$, $c''\in K$.  Then the inverse image of $K$ in $P$ is
elementary abelian of order $2|K|$.  The splittings $q:K\to P$ form a coset of
$\operatorname{Hom}(K,\{\pm1\})$, and may be chosen with any prescribed
lift over one nonzero word of $K$.  Every subcode of $B_O$ satisfies
\eqref{eq:5.0}.
\end{lemma}

\begin{proof}
By \eqref{eq:1.3}--\eqref{eq:1.5} the three conditions in \eqref{eq:5.0}
say precisely that every element of the inverse image squares to $1$, that
any two of them commute, and that any three of them associate.  A group of
exponent two is elementary abelian; a complement of $\langle-1\rangle$ in
$2^{k+1}$ is a hyperplane avoiding $-1$, and the set of these is a coset of
the character group of $K$.  For $K\le B_O$: the nonzero words of $B_O$ have
weight $8$ or $16$, so $|c|\equiv0\ (8)$; two distinct affine hyperplanes of
$X$ meet in $4$ or $0$ points and each meets $X$ in $8$, so
$|c\cap c'|\equiv0\ (4)$; and a nonempty intersection of at most three affine hyperplanes in
$X\cong\F_2^4$ has dimension at least one, hence even cardinality
(the empty intersection has cardinality zero).
\end{proof}

By Lemma~\ref{lem:elem-abelian} we may choose an elementary abelian
complement
\[
             Q_O=\{q_b:b\in B_O\}\cong 2^5
\]
in the inverse image of $B_O$ in $P$, calibrated by
\[
                         q_X=\Omega o,
\]
and we call the resulting basis of $A$ the \emph{calibrated basis} at $O$.
Write
\[
 \Hcal_O=B_O\setminus\{0,X\},\qquad y_b=x_{q_b}\quad(b\in\Hcal_O).
\]

For a character $\chi\in B_O^*$ set
\begin{equation}
 r_{O,\chi}=\frac12\left(
       -\sum_{i\in O}u_i+\sum_{i\in X}u_i
       +\vth\chi(X)x_o+\sum_{b\in\Hcal_O}\chi(b)y_b
                         \right).                                      \label{eq:5.1}
\end{equation}
Changing $o$ or $Q_O$ translates the character $\chi$ and changes the
chosen representatives by Parker signs; the resulting set of $\mu_3$-rays
is unchanged.

\begin{proposition}\label{prop:octadic-roots}
For every $\chi\in B_O^*$,
\begin{equation}
 (r_{O,\chi},r_{O,\chi})=9,
 \qquad r_{O,\chi}*r_{O,\chi}=10r_{O,\chi}.                 \label{eq:5.2}
\end{equation}
The cocode is transitive on the thirty-two roots over $O$, and $H$ is
transitive on all pairs $(O,\chi)$.
\end{proposition}

\begin{proof}
The restriction map $\C^*\to B_O^*$ is onto.  By
Corollary~\ref{cor:cocode}, a cocode element $\delta$ multiplies the
$y_b$-coefficient by $(-1)^{\langle b,\delta\rangle}$ and the
$\vth x_o$-coefficient by
\[
 (-1)^{\parity(\delta)+\langle O,\delta\rangle}
       =(-1)^{\langle X,\delta\rangle}.
\]
Thus the cocode is transitive on the thirty-two vectors in \eqref{eq:5.1}, and it
is enough to prove the assertions for $\chi=1$.

Put
\[
 U_O=-\sum_{i\in O}u_i+\sum_{i\in X}u_i,
 \qquad Y_O=\sum_{b\in\Hcal_O}y_b.
\]
Then $r_O=(U_O+\vth x_o+Y_O)/2$.  The three mutually orthogonal
summands have squared norms $3$, $3$, and $30$, respectively; hence
\[
                         \|r_O\|^2=\frac14(3+3+30)=9.
\]
For the product, it is useful to display the whole calculation.  Fix
$i\in O$, $j\in X$, and $c\in\Hcal_O$.  The entries below are the
coefficients before multiplication by the common outer factor $1/4$:
\begin{equation}
\begin{array}{c|rrrr}
 &u_i&u_j&x_o&y_c\\ \hline
 U_O*U_O&-\frac12&\frac72&0&0\\
 2U_O*(\vth x_o)&0&0&5\vth&0\\
 2U_O*Y_O&0&0&0&3\\
 (\vth x_o)*(\vth x_o)&-\frac92&\frac32&0&0\\
 2(\vth x_o)*Y_O&0&0&0&3\\
 Y_O*Y_O&-15&15&15\vth&14\\ \hline
 \text{sum}&-20&20&20\vth&20.
\end{array}                                                     \label{eq:5.2a}
\end{equation}
Here the first five rows follow immediately from
\eqref{eq:2.5}--\eqref{eq:2.9}.  For the last row, the thirty diagonal
terms $y_b*y_b$ give $-15$ on $O$ and $15$ on $X$, because each point of
$X\cong\F_2^4$ lies in fifteen affine hyperplanes.  The fifteen
complementary pairs $\{b,b+X\}$ give $15\vth x_o$.  Finally, for fixed
$c$, exactly twenty-eight ordered hyperplanes $b$ have
$b+c\in\Hcal_O$; since $Q_O$ is elementary abelian, their products all
have the same Parker sign and contribute $14y_c$.  The partner $c+X$ in
$2(\vth x_o)*Y_O$ contributes the remaining $3y_c$ in the preceding row.
Thus \eqref{eq:5.2a}, divided by $4$, is exactly $10r_O$.

The cocode transitivity already proved carries the norm and root equation
to every $r_{O,\chi}$.  Finally $M_{24}$ is transitive
on octads, and its lifts in $H$ may be adjusted by the cocode to carry any
phase to any other phase.  Hence $H$ is transitive on all pairs $(O,\chi)$.
\end{proof}

\begin{remark}\label{rem:Wilson-phase}
Replace $\vth$ in the complementary-octad line of \eqref{eq:2.9} by
$\xi\in\E$.  For the odd cocode transformation
$t_i=T_{\{i\}+\C}$, equation~\eqref{eq:3.8} gives
\[
 t_i(u_j)=u_j,
 \qquad t_i(x_d)=(-1)^{[i\in\bar d]}x_d,
\]
with conjugate-linear scalar action.  If $D$, $F$ are disjoint octads and
$G=(D\cup F)^c$, invariance of their product is equivalent to
\[
 \bar\xi(-1)^{[i\in G]}
   =\xi(-1)^{[i\in D]+[i\in F]}.
\]
Since $[i\in G]=1+[i\in D]+[i\in F]\pmod2$, this forces
$\bar\xi=-\xi$; conversely, that identity is sufficient for all basic
cocode transformations.  Thus the coefficient must be purely imaginary.
The root calculation \eqref{eq:5.2a} fixes its scale as well: with
$\xi$ in place of $\vth$, the $x_o$-coordinate would give
\[
                         5\vth+15\xi=20\vth,
\]
and hence $\xi=\vth$.

Thus the semilinear action and the octadic roots used here determine
the complementary-octad term as $\vth x_{df\Omega}$.  This fixes our
normalization of the Conway--Parker multiplication table; compare
\cite[Section~5.7.10]{WilsonBook}.
Replacing $(\Omega,\vth)$ simultaneously by $(-\Omega,-\vth)$ leaves
the algebra unchanged.
\end{remark}

For the trivial character write $r_O:=r_{O,1}$ and put
$\tau_O=\tau_{r_O}$.  Its block decomposition comes from the character spaces of an
$\E$-linear cocode subgroup.

Let
\begin{equation}
 V_O=\{S\subseteq O:|S|\text{ even}\}/\langle O\rangle,
 \qquad q([S])=\frac{|S|}{2}\pmod2.                              \label{eq:5.7}
\end{equation}
The polar form is $([S],[T])\mapsto|S\cap T|\pmod2$; its radical on
the even-weight code is $\langle O\rangle$.  Hence $V_O$ is
nondegenerate of dimension six.  Labeling $O$ by $\F_2^3$, the affine
functions form a doubly even four-dimensional subcode containing $O$;
its image in $V_O$ is totally singular of dimension three.  Thus $V_O$
has plus type.  Its $28$ nonsingular points are the duad classes, and
its $35$ nonzero singular points are the complementary tetrad pairs.

\begin{lemma}\label{lem:grading}
Put $N_O=B_O^\perp$ and
\[
 N_O^+=N_O\cap H^+=(B_O+\langle O\rangle)^\perp\cong2^6.
\]
The subgroup $N_O^+$ fixes $r_O$ and commutes with $\tau_O$.  Its
character group is naturally $V_O$, and its character spaces give
\begin{equation}
 A=\bigoplus_{[S]\in V_O}A_{[S]},\qquad
 A_{[S]}*A_{[T]}\subseteq A_{[S\triangle T]}.                    \label{eq:5.4}
\end{equation}
The coordinate vectors $u_i$ have character $[0]$, and $x_d$ has
character $[\bar d\cap O]$.  The $\E$-dimensions are
\begin{equation}
 \dim A_{[0]}=55,\qquad
 \dim A_{[p]}=16,\qquad \dim A_{[T]}=8,                         \label{eq:5.6}
\end{equation}
where $p$ is a duad and $T$ is a tetrad in $O$.
\end{lemma}
\begin{proof}
Since $X\in B_O$ and $\one=X+O$, even parity on $N_O$ is equivalent
to annihilating $O$.  Thus $\dim N_O^+=12-5-1=6$.
Its elements fix $U$, $x_o$, and all $y_b$, hence $r_O$; covariance
then gives commutation with $\tau_O$.

Restriction to $O$ identifies $\C/B_O$ with the even-weight code
on $O$, and therefore identifies $\C/(B_O+\langle O\rangle)$ with
$V_O$.  The perfect code--cocode pairing gives the asserted characters.
Because $N_O^+$ acts $\E$-linearly by algebra automorphisms, its
character spaces satisfy \eqref{eq:5.4}.  There are thirty octads
disjoint from $O$, sixteen meeting it in a prescribed duad, and
$\lambda_4-1=4$ meeting it in a prescribed tetrad.  Together with
$U$ and $x_o$, these give \eqref{eq:5.6} and
\[
                         55+28\cdot16+35\cdot8=783.
\]
\end{proof}

Restriction to $X$ induces an isometry
\begin{equation}
     V_O\cong RM(2,4)/RM(1,4),                                  \label{eq:5.8}
\end{equation}
where the quadratic form on the right is
$[f]\mapsto\operatorname{wt}(f)/2\pmod2$; all words of $RM(2,4)$
have even weight.  Indeed, restriction of $\C$ to $X$ has kernel $\langle O\rangle$
and hence image of dimension eleven.  Self-duality makes this image
orthogonal to $B_O=RM(1,4)$.  The eleven-dimensional space $RM(2,4)$
is this orthogonal complement: products of a quadratic monomial and
an affine function have degree at most three and sum to zero on
$\F_2^4$.  Quotienting by $B_O$ therefore gives \eqref{eq:5.8}.
The weight-half form descends modulo $RM(1,4)$, since its words have
weights divisible by four and are orthogonal to $RM(2,4)$.  The forms
agree because every $c\in\C$ has
$|c\cap O|+|c\cap X|\equiv0\pmod4$.

Modulo affine functions, every nonzero quadratic class is represented,
after a linear change of coordinates, by $v_1v_2$ or
$v_1v_2+v_3v_4$, of polar rank two and four, respectively.
These representatives have weights four and six, and the weight-half form
is invariant under linear coordinate changes.  Thus the rank is four
for a duad class, two for a tetrad class, and zero for the zero class.

For reference, substitution of the multiplication table into
\eqref{eq:4.2} gives the following formulas.  Put
\[
 s_O=\sum_{i\in O}u_i,\qquad s_X=\sum_{i\in X}u_i.
\]
Then
\begin{align}
 \tau_O(u_i)&=u_i-\frac3{16}s_O-\frac{\vth}{16}x_o,
                                      &&i\in O,                 \label{eq:5.10}\\
 \tau_O(x_o)&=-\frac{\vth}{2}s_O-\frac12x_o,                  \label{eq:5.11}\\
 \tau_O(u_i)&=\frac1{16}\left(s_X-
          \sum_{b\in\Hcal_O}\ev_i(b)y_b\right),
                                      &&i\in X,                 \label{eq:5.12}\\
 \tau_O(y_b)&=\frac12\left(
       \sum_{i\in b}u_i-\sum_{i\in X\setminus b}u_i
       +y_b+y_{b+X}\right), &&b\in\Hcal_O.                    \label{eq:5.13}
\end{align}
If $D$ is an octad with $|D\cap O|=2$ or $4$, then, with Parker signs
understood,
\begin{align}
 \tau_O(x_d)={}&\frac{3-|D\cap O|}{4}x_d
   -[|D\cap O|=4]\frac{\vth}{4}x_{do}                 \notag\\
 &+\frac14\sum_{\substack{b\in\Hcal_O\\|D\cap b|=4}}x_{dq_b}
  +\frac{\vth}{4}
       \sum_{\substack{b\in\Hcal_O\\D\cap b=\varnothing}}
                 x_{dq_b\Omega}.                               \label{eq:5.14}
\end{align}

Set
\[
\begin{aligned}
 A_O^{(9)}&=\operatorname{span}_{\E}
               \{u_i\ (i\in O),\ x_o\},\\
 A_X^{(46)}&=\operatorname{span}_{\E}
               \{u_j\ (j\in X),\ y_b\ (b\in\Hcal_O)\}.
\end{aligned}
\]

Thus $A_{[0]}=A_O^{(9)}\oplus A_X^{(46)}$.  For a nonzero class
$[S]$, write $A_p^{(16)}=A_{[p]}$ for a duad $p$ and
$A_T^{(8)}=A_{[T]}$ for a tetrad $T$.

\begin{proposition}[Octadic involution]\label{prop:octadic-involution}
The map $\tau_O$ is a conjugate-linear antiunitary involution.  Its support
matrix has the orthogonal decomposition
\begin{equation}
\begin{aligned}
 A={}&A_O^{(9)}\oplus A_X^{(46)}\\
   &\oplus\bigoplus_{p\in\binom O2}A_p^{(16)}\\
   &\oplus\bigoplus_{\{T,O\setminus T\}}A_T^{(8)}.
\end{aligned}                                                   \label{eq:5.15}
\end{equation}
The superscripts are $\E$-dimensions.  In suitable signed Parker
bases, each $16$-block is $H_{16}/4$ for a Hadamard matrix $H_{16}$,
and each $8$-block has conjugate-linear matrix
\begin{equation}
 \frac14\bigl((J_4-2I_4)\otimes J_\vth\bigr),\qquad
 J_\vth=\begin{pmatrix}1&\vth\\\vth&1\end{pmatrix},                \label{eq:5.16}
\end{equation}
where $J_4$ is the all-one matrix.
By $H$-covariance, every map $\tau_{r_{F,\chi}}$ attached to an octadic
root is likewise a conjugate-linear antiunitary involution.
\end{proposition}

\begin{proof}
The grading \eqref{eq:5.4} reduces the calculation to these blocks;
the displayed formulas split the trivial-character block into the two
exceptional summands.  On $A_O^{(9)}$, the map $\tau_O$ fixes every difference
$u_i-u_j$ and acts on $\langle s_O,x_o\rangle$ by the anti-linear matrix
$-J_\vth/2$; hence
$\overline{(-J_\vth/2)}(-J_\vth/2)=I$.

On $A_X^{(46)}$, put
\[
                    w_b=-\sum_{j\in X}\ev_j(b)u_j.
\]
Character orthogonality gives
\[
 \sum_{j\in X}\ev_j(b)\ev_j(c)
    =16(\delta_{b,c}-\delta_{b+X,c}).
\]
Equations \eqref{eq:5.12}--\eqref{eq:5.13} therefore show that
$\tau_O$ fixes $s_X$ and every $y_b+y_{b+X}$, while
\[
 \tau_O(w_b)=y_b-y_{b+X},\qquad
 \tau_O(y_b-y_{b+X})=w_b.
\]
These vectors decompose the exceptional block, and the paired vectors have
the same norm.  Thus both exceptional blocks are antiunitary involutions.

For a duad block, fix an octad $D$ with $D\cap O=p$ and a lift $d$.
Identify $X$ with $V=\F_2^4$, choosing the origin outside $D\cap X$,
and let $Q_D$ be the indicator of $D\cap X$.
By \eqref{eq:5.8}, $Q_D$ has weight six, $Q_D(0)=0$, and nondegenerate
polar form $\beta(u,v)=Q_D(u+v)+Q_D(u)+Q_D(v)$.
For $u\in V$ put
\[
              b_u(v)=\beta(u,v)+Q_D(u),\qquad d_u=dq_{b_u}.
\]
Since $Q_D(v)+b_u(v)=Q_D(v+u)$, the $X$-parts of $\overline{d_u}$
are the sixteen distinct translates of $D\cap X$.  Thus these are
exactly the sixteen octads meeting $O$ in $p$.

In the signed basis $x_{d_u}$, formula~\eqref{eq:5.14} has matrix
\[
             \frac14 H_{16},\qquad
             (H_{16})_{u,w}=(-1)^{\beta(u,w)}.
\]
To verify the signs, for $u\ne w$ the hyperplane $b_u+b_w$ meets
$\overline{d_u}$ in four points.  As $|D\cap b_u|$ is even, the
Parker associator and $q_bq_{b'}=q_{b+b'}$ give
\[
 (dq_{b_u})q_{b_u+b_w}
      =(-1)^{|D\cap b_u\cap b_w|}dq_{b_w}.
\]
The exponent is
\[
 \sum_{v\in V}Q_D(v)b_u(v)b_w(v)
   =\sum_{v\in V}Q_D(v)\beta(u,v)\beta(w,v)=\beta(u,w).
\]
Indeed every polynomial of degree at most three sums to zero over
$V$; with quadratic part $v_1v_2+v_3v_4$, the remaining degree-four
coefficient is precisely $\beta(u,w)$.  The diagonal coefficient is
also one.  Hence $H_{16}$ is a character table, and character
orthogonality gives $H_{16}H_{16}^t=16I$.

For a tetrad block, fix $D\cap O=T$ and choose affine coordinates on
$X$ with $Q_D(v)=[v\in D\cap X]=v_1v_2$.  The four octads meeting $O$
in $T$ have restrictions $Q_D+b$, where
\[
 b\in\{0,v_1,v_2,v_1+v_2+1\}.
\]
Use the paired lifts $dq_b$, $(dq_b)o$ for $T$ and $O\setminus T$.
Each $b$ is constant on the affine plane $D\cap X$, so the relevant
Parker associators have even exponent.  Substitution in
\eqref{eq:5.14}, using $q_X=\Omega o$, now gives exactly
\[
                \frac14\bigl((J_4-2I_4)\otimes J_\vth\bigr),
\]
where $J_4$ is the all-one matrix.  Within each tetrad half, the
coefficients are $-1$ on the diagonal and $1$ off it; between the two
halves, they are $-\vth$ on matched pairs and $\vth$ on the other pairs.  Since $(J_4-2I_4)^2=4I_4$ and
$J_\vth\overline{J_\vth}^{\,t}=4I_2$, this block is antiunitary.
This proves antiunitarity on every block.  Lemma~\ref{lem:symmetric-root-maps}
therefore gives $\tau_O^2=1$.
Proposition~\ref{prop:octadic-roots} and
Lemma~\ref{lem:root-covariance} transport these properties to every
octadic root.
\end{proof}

\subsection{A cubic-tensor criterion for root reflections}

Work in an orthonormal complex basis of $A_{\mathbb C}$, and write
\[
 C_{ijk}=\Phi(e_i,e_j,e_k),\qquad
 \|C\|^2=\sum_{i,j,k}|C_{ijk}|^2.
\]
All tensor contractions in this subsection use this Hermitian metric.
Define the symmetric trilinear tensor
\begin{equation}
 K_{pqr}=\sum_{i,j,k,a,b,c}
   \overline{C_{ijk}}\,\overline{C_{abc}}
                 C_{aip}C_{bjq}C_{ckr}.                     \label{eq:quintic-tensor}
\end{equation}
The expression is independent of the orthonormal basis: each internal
index occurs once in a tensor and once in a conjugate tensor.

For $x\in A$ write $L_x(y)=y*x$; although $L_x$ is conjugate-linear,
$L_xL_y$ is $\E$-linear.

\begin{proposition}[Trace form and cubic contraction]\label{prop:CP-tensor-identities}
The multiplication table of the Conway--Parker algebra satisfies
\begin{equation}
                   \|C\|^2=76734=783\cdot98,
                   \qquad K=1002C.                         \label{eq:CP-tensor-identities}
\end{equation}
Moreover the Hermitian form is recovered from multiplication by
\begin{equation}
           \operatorname{Tr}_{\E}(L_xL_y)=98(y,x).
           \label{eq:intrinsic-trace-form}
\end{equation}
\end{proposition}

\begin{proof}
For the norm identity, take $e_i=\sqrt8u_i$ on $U$ and the unit octad
vectors on $W$.  Put $\gamma=\sqrt2/64$ and $\beta=\sqrt2/8$.  The
nonzero coefficient types are
\begin{align}
 C_{ijk}&=\gamma\bigl(-1+16(\delta_{ij}+\delta_{jk}+\delta_{ki})
                         -128\delta_{ij}\delta_{jk}\bigr),
                                                    &&i,\,j,\,k\in I,\notag\\
 C_{iDD}&=\beta\bigl(4[i\in D]-1\bigr),                &&i\in I,\ D\in\Ocal,
                                                            \label{eq:orthonormal-cubic-table}\\
 C_{DEF}&=\tfrac12\sigma\quad\hbox{or}\quad
                      -\tfrac12\vth\sigma,\notag
\end{align}
Here a sextet triangle means three distinct octads with $D+E+F=0$;
they meet pairwise in four points.  The last alternative is a partition
of $I$ into three octads, or \emph{trio}; $\sigma$ is the corresponding
Parker sign.
The squared norm of the slice with first index $i\in I$ is
\[
 \frac{81^2+69\cdot15^2+506}{2048}
       +\frac{253\cdot9+506}{32}
          =\frac{353+2783}{32}=98.
\]
For a first index $D\in\Ocal$, it is
\[
             \frac{2(8\cdot9+16)}{32}
                    +\frac{280}{4}+\frac{30\cdot3}{4}=98.
\]
The distinct slices are orthogonal as well.  The even cocode characters
separate all octad slices from each other and from the point slices.
For two distinct point indices, the inner product of the slices is
\[
 \frac{2(-81)(15)+22\cdot15^2+2\cdot15^2-88\cdot15+22\cdot21}{2048}
       -\frac{33}{32}=0;
\]
here the last term uses
$\sum_D(4[i\in D]-1)(4[j\in D]-1)=-33$, obtained from
$\lambda_0=759$, $\lambda_1=253$, and $\lambda_2=77$.
Consequently $\sum_{a,b}C_{iab}\overline{C_{jab}}=98\delta_{ij}$.
Trace is unchanged by scalar extension, so this proves
\eqref{eq:intrinsic-trace-form} for $x$, $y\in A$; summing its diagonal gives
$\|C\|^2=783\cdot98$.

Both $C$ and $K$ are invariant under $H$, with the same conjugation
rule for semilinear elements.  The even cocode transformations force
the sum of the octad labels of a nonzero coefficient to be $0$ or
$\one$.  The remaining coefficient positions are therefore exactly
those in \eqref{eq:orthonormal-cubic-table}.  The coordinate action
identifies their ratios within seven types: three equality patterns
on $U$, two incidences of a point and an octad, and the two octad-triangle
types.  The odd cocode transformations impose the same real or
$\vth$-phase as in $C$.  Appendix~\ref{app:tensor-contractions} evaluates these contractions
from the Witt design and the Parker associator.  In particular, the two
signed octad sums are derived there by the six-tetrad description of the
Golay code.
Every one of the seven ratios is $1002$.
\end{proof}

\begin{corollary}\label{cor:intrinsic-form}
Every $\E$-linear or $\E$-conjugate-linear algebra automorphism of $A$
is respectively unitary or antiunitary.
\end{corollary}
\begin{proof}
For an automorphism $g$, multiplicativity gives
\[
                 L_{g(x)}L_{g(y)}=gL_xL_yg^{-1}.
\]
The trace on the right is $\operatorname{Tr}_{\E}(L_xL_y)$ when $g$ is
$\E$-linear, and its conjugate otherwise.
Apply \eqref{eq:intrinsic-trace-form}.
\end{proof}

\begin{theorem}[Tensor criterion]
\label{thm:root-tensor-criterion}
Let $r\in A$ satisfy $(r,r)=9$ and $r*r=10r$.  If $\tau_r$ is
antiunitary, then it is an involutory algebra automorphism.
\end{theorem}

\begin{proof}
Regard $r$ as a column vector in the orthonormal basis above.  Write
\[
 (R_r)_{ij}=\sum_p\overline{C_{ijp}}\,\bar r_p,
       \qquad P_r=rr^t,\qquad T_r=R_r-P_r.
\]
Thus $x*r=R_r\bar x$ and $\tau_r(x)=T_r\bar x$.  Both $R_r$ and $T_r$ are
symmetric.  Lemma~\ref{lem:symmetric-root-maps} gives $\tau_r^2=1$,
so $T_r\overline{T_r}=I$.  The root equation is $R_r\bar r=10r$; expanding
$T_r\overline{T_r}=I$ gives
\begin{equation}
                         R_r\overline{R_r}=I+11r\bar r^{\,t}.           \label{eq:root-R-square}
\end{equation}
On $A_{\mathbb C}$ its eigenvalues are $100$ on $\mathbb Cr$ and $1$ on
the orthogonal complement.

Let $C^{\tau_r}$ be the tensor of
$\overline{\Phi(\tau_rx,\tau_ry,\tau_rz)}$.  It has the same norm as
$C$.  The Hermitian inner product of these two tensors is
\[
 M=\sum_{i,j,k,a,b,c}C_{ijk}C_{abc}(T_r)_{ai}(T_r)_{bj}(T_r)_{ck}.
\]
Expand each $T_r=R_r-P_r$.  The term with no $P_r$ is
$\overline{K(r,r,r)}=1002\cdot90=90180$ by
Proposition~\ref{prop:CP-tensor-identities}; this scalar value is the
only use of $K=1002C$ in the argument.  The identity supplies this value
for every root $r$ with antiunitary root map, including roots not yet
known to lie in the constructed families.  Each term with one $P_r$ is
\[
             \operatorname{tr}\bigl((R_r\overline{R_r})^2\bigr)
                               =10000+782=10782.
\]
Each term with two $P_r$'s is
$10^2\bar r^{\,t}R_r\bar r=10^3(r,r)=9000$, and the term with three is
$\Phi(r,r,r)^2=90^2=8100$.  Therefore
\begin{equation}
 M=90180-3\cdot10782+3\cdot9000-8100
                                =76734=\|C\|^2.              \label{eq:cubic-defect-zero}
\end{equation}
It follows that $\|C^{\tau_r}-C\|^2=2\|C\|^2-2\operatorname{Re}M=0$.
Thus $\Phi(\tau_rx,\tau_ry,\tau_rz)=\overline{\Phi(x,y,z)}$.
Together with antiunitarity and nondegeneracy of the Hermitian form,
this is precisely $\tau_r(x*y)=\tau_r(x)*\tau_r(y)$.
\end{proof}

\begin{corollary}\label{cor:octadic-automorphism}
Every octadic root map $\tau_{r_{F,\chi}}$ is an antiunitary involutory
algebra automorphism.
\end{corollary}

\begin{proof}
Apply Theorem~\ref{thm:root-tensor-criterion} to
Propositions~\ref{prop:octadic-roots} and~\ref{prop:octadic-involution}.
\end{proof}

\subsection{The absolute bound for reflecting rays}

Call a root \emph{reflecting} if its root map is antiunitary.  By
Theorem~\ref{thm:root-tensor-criterion}, it is then an involutory algebra
automorphism.

\begin{lemma}
\label{lem:reflecting-root-pairing}
If $r$, $s$ are reflecting roots on distinct rays, then
\begin{equation}
                              (r,s)\in\{0\}\cup\mu_3.         \label{eq:forced-root-pairing}
\end{equation}
\end{lemma}

\begin{proof}
Put $c=(r,s)$ and $t=r*s$.  From $\tau_r^2(s)=s$, $\tau_r(r)=r$, and
$(r,t)=\Phi(r,r,s)=10\bar c$, one obtains
\[
                t*r=s+11\bar c\,r,\qquad t*s=r+11c\,s.
\]
Because $\tau_r$ preserves the root equation of $s$, the vector
$\tau_r(s)=t-cr$ has square $10(t-cr)$.  Expanding this equation gives
\[
        t*t=10t+2\bar c\,s+(12\bar c^{\,2}-10c)r.
\]
Interchanging $r$ and $s$ gives
\[
        t*t=10t+2c\,r+(12c^2-10\bar c)s.
\]
Subtract the two expressions:
\begin{equation}
              (\bar c^{\,2}-c)r+(\bar c-c^2)s=0.             \label{eq:pairing-polynomial}
\end{equation}
The vectors are linearly independent.  Indeed, proportional norm-$9$
roots differ by a scalar of modulus one, and their square equations
force that scalar to have cube one.  Therefore $\bar c^{\,2}=c$ and
$c^2=\bar c$.  Either $c=0$, or $c^3=1$, as asserted.
\end{proof}

\begin{lemma}[Quadratic bound]\label{lem:quadratic-unisolvence}
Let $v_1$, $\ldots$, $v_N$ lie in a positive-definite Hermitian space of
complex dimension $m$, with $(v_i,v_i)=a>1$ and
$(v_i,v_j)\in\{0\}\cup\mu_3$ for $i\ne j$.  Then their symmetric
squares are linearly independent, and
\begin{equation}
                              N\leq\binom{m+1}{2}.            \label{eq:quadratic-absolute-bound}
\end{equation}
\end{lemma}
\begin{proof}
Let $G=((v_i,v_j))$.  The Gram matrix of the symmetric squares is
\[
                  G\circ G=(a^2-a)I+\overline G>0,
\]
since $z^2=\bar z$ for $z\in\{0\}\cup\mu_3$.  This proves independence
in $\operatorname{Sym}^2(\mathbb C^m)$ and hence the bound.
\end{proof}

Applied to reflecting roots in $A\subset A_{\mathbb C}$, these lemmas
give at most
$\binom{784}{2}=306936$ rays.  We next construct enough duadic roots
to attain this bound.

\subsection{Duadic roots from one chosen octad pair}

For a duad $p\subset I$ put
\begin{equation}
 \C_p=\{c\in\C:c\cap p=\varnothing\},\qquad
 \Xi_p=\{c\in\C_p:|c|=8\text{ or }16\}.                    \label{eq:6.1}
\end{equation}
Then $\dim\C_p=10$, and $\Xi_p$ contains $330$ words of weight $8$ and
$77$ of weight $16$.  Write $\C_p^*=\operatorname{Hom}(\C_p,\{\pm1\})$.
For the moment choose just one pair of octads $F$, $F'$ with $F\cap F'=p$.
No independence of this choice is assumed.

\begin{lemma}\label{lem:shortened-duad}
For such a pair,
\begin{equation}
                          \C_p=B_F\oplus B_{F'}.              \label{eq:6.2}
\end{equation}
\end{lemma}
\begin{proof}
Both summands have dimension five.  A nonzero word in their intersection
would be an octad $D$ supported on the ten-set $I\setminus(F\cup F')$.
Then $F+F'+D$ would have weight $20$, which does not occur in the Golay
code.  The intersection is zero, proving the assertion.
\end{proof}

Fix a section $s:\C\to P$ and put
\begin{equation}
 z_c=\begin{cases}x_{s(c)},&|c|=8,\\
                  \vth x_{s(c)\Omega},&|c|=16.
       \end{cases}                                           \label{eq:6.4}
\end{equation}
For a cocode element $\delta$,
\begin{equation}
                    T_\delta(z_c)=(-1)^{\langle c,\delta\rangle}z_c.
                                                               \label{eq:6.5}
\end{equation}
In the second case conjugation of $\vth$ cancels the extra character
of the all-one word.

\begin{lemma}\label{lem:no-octad-dodecad}
No octad is contained in a Golay dodecad.  If $F\cap F'$ is a duad,
there is no octad disjoint from $F\cup F'$.
\end{lemma}
\begin{proof}
Subtracting a contained octad from a dodecad would give a word of weight
four.  For the second assertion the complement of $F+F'$ is a dodecad
containing $I\setminus(F\cup F')$.
\end{proof}

\begin{lemma}\label{lem:orthogonal-octadic-pairs}
If $F\cap F'=p$, then $(r_{F,\chi},r_{F',\chi'})=0$ for all phases.
\end{lemma}
\begin{proof}
The $U$-parts pair to $(2-6-6+10)/32=0$.  The distinguished octad
coordinates occur in neither opposite root, and a common remaining
coordinate would be an octad disjoint from $F\cup F'$, excluded by the
preceding lemma.
\end{proof}

\begin{proposition}[Duadic fibre]\label{prop:chosen-duadic-fibre}
The $32^2$ products $r_{F,\chi}*r_{F',\chi'}$ are distinct roots.  For a
sign function $\eta_p:\Xi_p\to\{\pm1\}$ determined by the chosen pair,
they are exactly
\begin{equation}
 r_{p,\psi}=\frac18\left(U_p+
        \sum_{c\in\Xi_p}\eta_p(c)\psi(c)z_c\right),\qquad
 U_p=-7\sum_{i\in p}u_i+\sum_{i\notin p}u_i,
 \quad\psi\in\C_p^*.                                      \label{eq:6.22}
\end{equation}
Every map $\tau_{r_{p,\psi}}$ is an antiunitary involutory algebra
automorphism.  The cocode acts transitively on this fibre, with kernel of
order four.
\end{proposition}

\begin{proof}
Write $r_{F,\chi}=\tfrac12(U_F+\sum_{0\ne b\in B_F}\chi(b)w_F(b))$,
where $w_F(b)=\vth^{\nu(b)}x_{q_b\Omega^{\nu(b)}}$ and $\nu(b)=0$, $1$
according as $|b|=8$, $16$.  The two $W$-supports are disjoint.  Consequently
only $\tfrac14U_F*U_{F'}$ contributes to $U$, and the coordinate
multiplication table gives $U_p/8$.

For a coordinate $z_c$, use its unique decomposition $c=b+b'$ in
\eqref{eq:6.2}.  If both components are nonzero, there is exactly one
$W$--$W$ contribution, and its coefficient is
$\pm\chi(b)\chi'(b')/8$.  To check the magnitude, two weight-$16$
components have a dodecad as their sum and contribute no coordinate.
A weight-$16$ component and a weight-$8$ component cannot have a
weight-$8$ sum: the latter component would then avoid both $F$ and $F'$.
In the remaining cases the factor $\vth$ in the product is exactly the
factor prescribed by $z_c$, leaving magnitude $1/8$.

If $b=0$, the contribution is $\tfrac14U_F*w_{F'}(c)$.
For a supporting octad $D$ the scalar in $U_F*x_D$ is
$(3-|D\cap F|)/2$.  Here $|D\cap F|=2$ or $4$: intersections zero and
eight are excluded by the same no-octad-in-a-dodecad argument.
The magnitude is again $1/8$ after division by four, including the
conjugation of $\vth$.  The case $b'=0$ is identical.  Codewords of
weight $12$ label no $W$-coordinate, so there are no further terms.

Thus the phase dependence is $\chi(b)\chi'(b')$.  The direct sum
\eqref{eq:6.2} identifies $B_F^*\times B_{F'}^*$ with $\C_p^*$, proving
\eqref{eq:6.22}.  The labels $\Xi_p$ span $\C_p$, since they contain
the nonzero words of both summands; hence the $1024$ characters give
$1024$ distinct vectors.

By orthogonality each product is $\tau_{r_{F,\chi}}(r_{F',\chi'})$.
Corollary~\ref{cor:octadic-automorphism} gives its norm and root equation, and
covariance gives
\[
 \tau_{r_{p,\psi}}=
       \tau_{r_{F,\chi}}\tau_{r_{F',\chi'}}\tau_{r_{F,\chi}}.
\]
Finally \eqref{eq:6.5} gives all character translations of $\C_p^*$;
the annihilator of its ten-dimensional code has order $2^{12-10}=4$.
\end{proof}

Make one octad-pair
choice for each duad and define
\[
 \Rcal=\Rcal_{\rm b}\sqcup\Rcal_{\rm o}\sqcup\Rcal_{\rm d}
       =\{[r_i]\}\sqcup\{[r_{O,\chi}]\}\sqcup\{[r_{p,\psi}]\}.
\]
The $U$-coordinates distinguish the three shapes and their supports,
and their nonzero real entries distinguish them even up to $\mu_3$.
Thus
\begin{equation}
 |\Rcal|=24+759\cdot32+\binom{24}{2}\,1024
                     =306936=\binom{784}{2}.                 \label{eq:7.1}
\end{equation}
All these are roots with antiunitary involutory algebra-automorphism
root maps, and their number attains the absolute bound.  Completeness
will now remove every octad-pair choice made in this construction.

\subsection{Extremal completion}\label{sec:extremal-completion}

\begin{theorem}[Extremal closure]\label{thm:extremal-closure}
The set $\Rcal$ is exactly the set of rays of reflecting
roots.  In particular every semilinear isometric algebra automorphism
of $A$ permutes $\Rcal$; this includes $H$ and every $\tau_r$ with $[r]\in\Rcal$.
\end{theorem}
\begin{proof}
All $306936$ rays in $\Rcal$ are reflecting by the preceding construction.
If a further reflecting ray existed, Lemma~\ref{lem:reflecting-root-pairing}
would give a set of $306937$ representatives with the pairings required
in Lemma~\ref{lem:quadratic-unisolvence}.  This contradicts
\[
                          306936=\binom{783+1}{2}.
\]
A semilinear isometric algebra automorphism preserves the norm and root
equation.  Covariance of root maps shows that it also preserves the
reflecting property.  It therefore permutes precisely $\Rcal$.
\end{proof}

\begin{corollary}\label{cor:duadic-phase}
For every pair of octads $F$, $F'$ with $F\cap F'=p$, the $32^2$ products
$r_{F,\chi}*r_{F',\chi'}$ are exactly the $1024$ roots
$r_{p,\psi}$ of \eqref{eq:6.22}.  Thus this fibre is independent of the
chosen octad pair.  Its coordinate shape is
\begin{equation}
 \frac18\bigl(-7^2,1^{22}\mid
                    \pm\vth^{77},\ \pm1^{330},\ 0^{352}\bigr).\label{eq:6.24}
\end{equation}
\end{corollary}
\begin{proof}
Apply Proposition~\ref{prop:chosen-duadic-fibre} to the new pair.
Its $1024$ distinct products are reflecting roots and so belong to
$\Rcal$.  Their $U$-part is $U_p/8$, which distinguishes exactly the
already chosen fibre over $p$ and fixes its cube-root scalar.
Equality of the two fibres follows.  The coordinate multiplicities
are the Witt-design counts $77$, $330$, $352$.
\end{proof}

\begin{remark}
The theorem classifies reflecting roots, not arbitrary solutions of
$(r,r)=9$, $r*r=10r$.
\end{remark}

\begin{lemma}\label{lem:intrinsic-support}
For $i\ne j$ and for octadic and duadic roots one has
\begin{align}
 (r_i,r_j)&=1,                                  \label{eq:9.2a}\\
 (r_i,r_{O,\chi})
   &=\begin{cases}1,&i\in O,\\0,&i\notin O,\end{cases}          \label{eq:9.2}\\
 (r_i,r_{p,\psi})
   &=\begin{cases}1,&i\in p,\\0,&i\notin p.\end{cases}        \label{eq:9.7}
\end{align}
Thus the octad or duad support of a nonbasic ray is recovered intrinsically
as the set of basic rays having nonzero pairing with it.
\end{lemma}

\begin{proof}
All three identities follow immediately from \eqref{eq:4.3},
\eqref{eq:5.1}, and \eqref{eq:6.22}; the phase-dependent $W$-parts are
orthogonal to every basic root.
\end{proof}

\begin{lemma}\label{lem:root-spanning}
The basic and octadic roots span $A$.  More precisely, every coordinate
generator is a linear combination of these roots with coefficients in
$\mathbb Z[\omega,1/2]$.
\end{lemma}
\begin{proof}
The basic-root formulas give
\begin{equation}
 u_i=\frac1{128}\sum_{j\in I}r_j-\frac18r_i.
 \label{eq:basic-coordinate-extraction}
\end{equation}
For an octad $D$, choose an octad $F$ disjoint from $D$; there are
$30$ choices by Lemma~\ref{lem:Witt-numerology}.  Character orthogonality
in the calibrated fibre at $F$ gives
\begin{equation}
 \sum_{\chi\in B_F^*}\chi(D)r_{F,\chi}=16y_D,
 \qquad y_D=\pm x_D.
 \label{eq:CP-coordinate-Fourier}
\end{equation}
Thus all $24+759$ coordinate generators have the asserted expressions.
\end{proof}

\section{The three-transposition group and its frame geometry}
\label{sec:three-transposition}
\subsection{The faithful three-transposition action}

Let
\[
       \widehat G=\langle\tau_r:[r]\in\Rcal\rangle,
\]
using the chosen coordinate representatives.  Extremal closure makes this a
group of semilinear isometric algebra automorphisms preserving $\Rcal$.

By \eqref{eq:root-scalar} and
Lemma~\ref{lem:scalar-automorphisms}, a ray determines a projective root
involution.  Denote its induced permutation of $\Rcal$ by $d_{[r]}$,
and put
\[
       Z=\mu_3\operatorname{id}_A,
       \qquad
       \Dcal_{\rm CP}=\{d_{[r]}:[r]\in\Rcal\}.
\]
This permutation is also the action of the projective class $\tau_rZ$.
Lemma~\ref{lem:root-covariance} descends to the projective covariance
identity
\begin{equation}
                  g d_{[r]}g^{-1}=d_{[g(r)]}                  \label{eq:projective-covariance}
\end{equation}
for every semilinear algebra automorphism $g$ preserving $\Rcal$.
For later use write
\[
              d_i:=d_{[r_i]},\qquad
              d_{O,\chi}:=d_{[r_{O,\chi}]},\qquad
              d_{p,\psi}:=d_{[r_{p,\psi}]}.
\]
Define the faithful ray permutation group
\begin{equation}
 G_{\rm CP}:=\langle\Dcal_{\rm CP}\rangle
       \leq\operatorname{Sym}(\Rcal).                         \label{eq:9.5a}
\end{equation}

\begin{lemma}[Reflection calculus]\label{lem:reflection-calculus}
Let $[r]$, $[s]\in\Rcal$ be distinct rays.  Choose the canonical root
representatives and put $c=(r,s)$.
\begin{enumerate}[label=(\alph*)]
\item If $c=0$, then $t=r*s=\tau_r(s)$ is a root and
\[
       \tau_t=\tau_r\tau_s\tau_r=\tau_s\tau_r\tau_s.
\]
Hence $d_{[r]}d_{[s]}$ has order $3$.
\item If $c\in\mu_3$, then
\begin{equation}
 \tau_r(s)=\bar c\,s,\quad r*s=cr+\bar c\,s,
 \qquad(\tau_r\tau_s)^2=c\operatorname{id}_A.                 \label{eq:9.6}
\end{equation}
Hence $d_{[r]}$ and $d_{[s]}$ commute.
\end{enumerate}
\end{lemma}

\begin{proof}
Assume first that $c=0$ and put $t=\tau_r(s)=r*s$.  Since $\tau_r$ is an algebra automorphism preserving $\Rcal$,
Lemma~\ref{lem:root-covariance} gives
\[
       \tau_t=\tau_r\tau_s\tau_r;
\]
commutativity gives the symmetric braid expression.  Moreover
\[
 (s,t)
  =(s,r*s)
  =(r,s*s)
  =10(r,s)=0.
\]
Thus $[t]\ne[s]$.  The involution $d_{[r]}$ moves $[s]$ to $[t]$, whereas
$d_{[s]}$ fixes $[s]$; hence $d_{[r]}\ne d_{[s]}$.  The braid relation
therefore makes $d_{[r]}d_{[s]}$ an element of exact order $3$.

The first assertion in (b) is
Lemma~\ref{lem:nonorthogonal-rigidity}; the remaining assertions follow
from Lemma~\ref{lem:root-covariance} and \eqref{eq:root-scalar}.
\end{proof}

\begin{lemma}\label{lem:basic-fixed-ray}
For any $[r]$, $[s]\in\Rcal$,
\begin{equation}
       d_{[r]}([s])=[s]\quad\Longleftrightarrow\quad(r,s)\ne0.
                                                        \label{eq:basic-fixed-ray}
\end{equation}
Moreover, for every $i\in I$,
\begin{equation}
                    d_{[r]}=d_i\quad\Longleftrightarrow\quad[r]=[r_i].
                                      \label{eq:basic-transposition-separation}
\end{equation}
\end{lemma}
\begin{proof}
The equal-ray case is immediate.  For distinct rays the preceding
lemma gives the fixed-ray assertion when $(r,s)\ne0$; when $(r,s)=0$,
it gives $(s,\tau_r(s))=0$, so the ray moves.

If $d_{[r]}=d_i$, then $[r]$ pairs nontrivially with every basic ray,
since $d_i$ fixes all of them.  The coordinate supports in
Lemma~\ref{lem:intrinsic-support} force $[r]=[r_j]$ for some $j$.
For $j\ne i$, an octad containing $i$ but not $j$ distinguishes
$d_i$ from $d_j$ by the fixed-ray assertion; such octads exist since
$253-77>0$.  Thus $j=i$.
\end{proof}

\begin{proposition}
\label{prop:one-class}
The zero-pairing graph on $\Rcal$ is connected, $G_{\rm CP}$ is
transitive on $\Rcal$, and $[r]\mapsto d_{[r]}$ is injective.
The projective root reflections form one conjugacy class.
\end{proposition}
\begin{proof}
Every nonbasic ray has a basic point outside its support and is
therefore joined to a basic ray in the zero-pairing graph.  Any two
basic rays have an octad avoiding both coordinates: there are $330$
such octads by Lemma~\ref{lem:Witt-numerology}.  An octadic ray over
one of them joins the two basic rays by a path of length two.
Thus the graph is connected.
Along any edge, putting $t=r*s$ gives
\[
                  d_{[s]}([r])=[t],\qquad d_{[r]}([t])=[s].
\]
Its two endpoints are therefore in one $G_{\rm CP}$-orbit.
Connectivity proves transitivity.

If $d_{[r]}=d_{[s]}$, carry $[r]$ to a basic ray by this transitivity.
Covariance and Lemma~\ref{lem:basic-fixed-ray} force the image of $[s]$
to be the same basic ray.  Hence $[r]=[s]$.  Transitivity and covariance
also prove the conjugacy-class assertion.
\end{proof}

\begin{corollary}\label{cor:uniform-valencies}
Every ray has $31671$ distinct nonorthogonal partners and $275264$
orthogonal partners.
\end{corollary}
\begin{proof}
By transitivity it suffices to count at a basic ray $[r_i]$.
Lemma~\ref{lem:intrinsic-support} gives $23$ other basic rays,
$32$ rays over each of the $253$ octads through $i$, and $1024$ rays
over each of the $23$ duads through $i$.  Thus the nonorthogonal
valency is
\[
             23+253\cdot32+23\cdot1024=31671,
\]
and the orthogonal valency is $306936-1-31671=275264$.
\end{proof}

\begin{theorem}[Three-transposition group]\label{thm:axis-group}
The pair
\[
       (G_{\rm CP},\Dcal_{\rm CP})
\]
is a finite centerless $3$-transposition group.  The generating set
$\Dcal_{\rm CP}$ is one conjugacy class of size $306936$.  Its diagram,
where adjacency means product order $3$, has degree $275264$; the commuting
codiagram has degree $31671$.
\end{theorem}

\begin{proof}
Lemma~\ref{lem:reflection-calculus} shows that products of distinct
members of $\Dcal_{\rm CP}$ have order $2$ or $3$, and
Proposition~\ref{prop:one-class} gives distinctness and the one-class
assertion.  The group is finite because it acts faithfully on the finite set
$\Rcal$.  If $z$ is central, then conjugation by $z$ fixes every member of
$\Dcal_{\rm CP}$.  Covariance gives
$d_{z[r]}=z d_{[r]}z^{-1}=d_{[r]}$; injectivity of the ray--reflection map
therefore implies that $z$ fixes every ray.  The action on $\Rcal$ is
faithful, so $z=1$.  The two degrees are the counts in
Corollary~\ref{cor:uniform-valencies}.
\end{proof}

\subsection{The scalar kernel and the perfect linear lift}

\begin{lemma}[Scalar kernel]\label{lem:scalar-kernel}
The group $\widehat G$ contains $Z$.  Every semilinear algebra
automorphism fixing each ray of $\Rcal$ belongs to $Z$.
In particular the kernel of $\widehat G$ on $\Rcal$ is $Z$, and
$\widehat G$ is finite.
\end{lemma}
\begin{proof}
There are chosen root representatives with nonreal pairing.  Indeed, for disjoint
octads $O$, $F$, character averaging in \eqref{eq:5.1} gives
\begin{equation}
 \frac1{32^2}\sum_{\chi\in B_O^*,\,\psi\in B_F^*}
       (r_{O,\chi},r_{F,\psi})
   =\left(\frac{U_O}{2},\frac{U_F}{2}\right)=-\frac14,
 \qquad U_O=-\sum_{i\in O}u_i+\sum_{i\notin O}u_i.
 \label{eq:nonreal-pairing-average}
\end{equation}
The allowed pairings are $0$, $1$, $\omega$, $\bar\omega$, so one of these
pairings is $\omega$ or $\bar\omega$.  Equation~\eqref{eq:9.6}
therefore puts a generator of $Z$ in $\widehat G$.

Let $g$ be a semilinear algebra automorphism fixing every ray.  It is
an isometry by Corollary~\ref{cor:intrinsic-form}.  Write
$g(r)=\lambda_r r$ on
the chosen coordinate representatives.  All $\lambda_r$ belong to $\mu_3$.
Whenever $(r,s)=1$, (anti)unitarity implies $\lambda_r=\lambda_s$.
The graph with edges $(r,s)=1$ is connected: the basic roots form a
clique and every other root is joined to a basic root by
Lemma~\ref{lem:intrinsic-support}.  Hence all $\lambda_r$ are a common
$\lambda$.  If $g$ is linear, Lemma~\ref{lem:root-spanning} gives $g=\lambda I$.
If $g$ is conjugate-linear, $\lambda^{-1}g$ fixes all chosen root representatives,
which would make all their pairings real, contradicting
\eqref{eq:nonreal-pairing-average}.  Thus the kernel is $Z$.
\end{proof}

\begin{proposition}[Parity and perfectness]\label{prop:parity-perfect}
Let $\widehat G^+$ be the $\E$-linear subgroup of $\widehat G$.
Then
\[
 \widehat G'=\widehat G^+,\qquad
 (\widehat G^+)'=\widehat G^+,
 \qquad [G_{\rm CP}:G_{\rm CP}']=2,
 \qquad (G_{\rm CP}')'=G_{\rm CP}'.
\]
There are exact sequences
\begin{align}
 1\longrightarrow Z\longrightarrow\widehat G
   \longrightarrow G_{\rm CP}\longrightarrow1,             \label{eq:12.1}\\
 1\longrightarrow Z\longrightarrow\widehat G^+
   \longrightarrow G_{\rm CP}'\longrightarrow1.            \label{eq:12.2}
\end{align}
The second is a nonsplit central extension.
\end{proposition}
\begin{proof}
Every generator $\tau_r$ is conjugate-linear.  Thus $\widehat G^+$
is the even-word subgroup, generated by all products $\tau_r\tau_s$.
For a path $r=r_0$, $r_1$, $\ldots$, $r_k=s$,
\[
 \tau_r\tau_s=(\tau_{r_0}\tau_{r_1})
             (\tau_{r_1}\tau_{r_2})\cdots
             (\tau_{r_{k-1}}\tau_{r_k}).
\]
The zero-pairing graph is connected, and its edge products have order
$3$ by Lemma~\ref{lem:reflection-calculus}(a).  The unit-pairing
graph is connected by the preceding proof, and its edge products have
order $2$ by \eqref{eq:9.6}.  Hence the abelianization of
$\widehat G^+$ has exponent dividing both $3$ and $2$, so is trivial.
Since $\widehat G/\widehat G^+\cong C_2$, this also gives
$\widehat G'=\widehat G^+$.

The scalar kernel is linear, so parity descends to a surjection
$G_{\rm CP}\to C_2$.  Its kernel is the perfect image of
$\widehat G^+$, hence equals $G_{\rm CP}'$.  This proves both exact
sequences and the remaining perfectness assertion.  The second
extension is central because its maps are linear; a splitting would
make the perfect group $\widehat G^+$ a direct product
$Z\times G_{\rm CP}'$, which has a nontrivial abelian quotient.
\end{proof}

\begin{definition}\label{def:CP-Fischer}
The permutation group $G_{\rm CP}$ of \eqref{eq:9.5a}, generated by the
projective reflections of the explicitly constructed root-ray set $\Rcal$, is
called the \emph{Conway--Parker Fischer group} and is also denoted
$G_{24}$.
\end{definition}

\subsection{The standard frame stabilizer}\label{sec:frame-stabilizer}

Let
\[
             \mathcal F_0=\{d_i:i\in I\},
             \qquad E_0=\langle\mathcal F_0\rangle.
\]
By Corollary~\ref{cor:relations},
\begin{equation}
                         E_0\cong\C^*\cong2^{12}.              \label{eq:standard-basic-group}
\end{equation}
Lemma~\ref{lem:scalar-kernel} makes the ray kernel of $H$ equal
to $H\cap Z=1$: a scalar map permuting the coordinate vectors $u_i$
must be the identity.  Thus $H$, and in particular its cocode subgroup,
acts faithfully on $\Rcal$.  Write $\overline H$ for its ray image.
It stabilizes $\mathcal F_0$ and induces the full Mathieu group on its
$24$ members.

For a nonbasic ray $[r]$, define its support relative to $\mathcal F_0$ by
\[
 \operatorname{supp}_{\mathcal F_0}([r])
   =\{i\in I:(r_i,r)\ne0\}.
\]
Lemma~\ref{lem:intrinsic-support} identifies these supports with the
Golay octads and duads.  In particular the octadic supports, together with
the $24$ basic points, recover the Witt design intrinsically from the
Hermitian ray configuration.

\begin{proposition}[Coordinate-basis stabilizer]
\label{prop:pointwise-U}
Let $g$ be an $\E$-linear or $\E$-conjugate-linear algebra automorphism
with $g(u_i)=u_i$ for all $i\in I$.  Then $g=T_\delta$ for a unique
$\delta\in\C^*$.  Its semilinearity is $\parity(\delta)$.
\end{proposition}

\begin{proof}
Multiplicativity makes $g$ commute with $L_{r_i}:x\mapsto x*r_i$.
By \eqref{eq:root-R-square},
$W=\bigcap_i\ker(L_{r_i}^2-I)$ because the basic roots span $U$.
Thus $g$ preserves $W$.
On $W$ these multiplication operators equal $t_i$.  The common
character spaces of the even cocode on $W$ are the individual octad
lines: two octads define the same character only when their difference
is $0$ or $\one$, and the latter cannot occur.
Hence $g(x_D)=\mu_Dx_D$.  Applying $g$ to
\eqref{eq:2.8} gives $\bar\mu_D^2=1$, so $\mu_D=\pm1$.
In particular $g$ is a semilinear isometry.

Fix a duad $p=\{i,j\}$ and one duadic root over $p$.
By extremal closure, $g$ permutes the roots in this fibre; their
$U$-part fixes the cube-root scalar.  The cocode is transitive on the
fibre, so after composing with a cocode element we may assume that
$g$ fixes this root.  Let $\epsilon$ record whether this new $g$ is
conjugate-linear.  The nonzero duadic coordinates then give
\begin{equation}
 \mu_D=1\quad(D\cap p=\varnothing),\qquad
 \mu_D=(-1)^\epsilon\quad(p\subset D).                     \label{eq:duad-fixed-signs}
\end{equation}

It remains to determine the signs on
$\mathcal A_i=\{D:i\in D,\ j\notin D\}$ and $\mathcal A_j$.
Join two members of $\mathcal A_i$ when their intersection has size
four.  Their sum is an octad disjoint from $p$, so
\eqref{eq:2.9} and \eqref{eq:duad-fixed-signs} force their signs to
agree.  This graph has $176$ vertices and degree $105$, hence is
connected: each component would have at least $106$ vertices.
Indeed, for $D\in\mathcal A_i$, there are $140$ other octads through
$i$ meeting $D$ in four points by \eqref{eq:Witt-point-inside}.
Among octads through $i$, $j$, the corresponding counts satisfy
$m_2+m_4=77$ and $m_2+3m_4=7\cdot21=147$, so $m_4=35$.
The degree is therefore $140-35=105$.
The same argument applies to $\mathcal A_j$.

Write the two constant signs as $(-1)^a$, $(-1)^b$.
There are $70-35=35$ octads $E\in\mathcal A_j$ meeting the fixed
$D\in\mathcal A_i$ in four points, by
\eqref{eq:Witt-point-refinement}.  Their sum contains $p$, so its
product relation gives $a+b=\epsilon$.  Thus on every octad line
\[
            \mu_D=(-1)^{a[i\in D]+b[j\in D]},
\]
and $g=T_{a\{i\}+b\{j\}+\C}$.  Undoing the initial cocode
composition proves the assertion.  Uniqueness follows from the
faithfulness of the cocode action.
\end{proof}

\begin{lemma}\label{lem:basic-ray-scalar-rigidity}
Let $g$ be an $\E$-linear unitary or $\E$-conjugate-linear antiunitary
algebra automorphism.  If $g$ fixes every basic ray $[r_i]$, then there are
unique $\lambda\in\mu_3$ and $\delta\in\C^*$ such that
\[
                              g=\lambda T_\delta.
\]
\end{lemma}

\begin{proof}
Write $g(r_i)=\lambda_i r_i$ with $\lambda_i\in\mu_3$.  Since
$(r_i,r_j)=1$ for $i\ne j$, (anti)unitarity gives
$\lambda_i\overline{\lambda_j}=1$; hence all $\lambda_i$ are one
common scalar $\lambda$.  The automorphism $\lambda^{-1}g$ fixes every
$r_i$.  The matrix relating the $r_i$ to the $u_i$ is $J-8I$, which is
invertible over $\mathbb Q$, so $\lambda^{-1}g$ fixes every $u_i$.  Proposition~\ref{prop:pointwise-U}
then gives $\lambda^{-1}g=T_\delta$ for a unique $\delta\in\C^*$.
Uniqueness of $\lambda$ follows from the action on $U$, and uniqueness of
$\delta$ from the Golay code--cocode pairing.
\end{proof}

\begin{proposition}[Octadic frame switch]
\label{prop:intrinsic-frame-switch}
For every octad $O$, the standard frame has two switched companions
obtained from the two parity classes of octadic phases over $O$.  More precisely, with the evaluation characters $\ev_j$ of
\eqref{eq:evaluation-character},
\begin{equation}
 \tau_O(r_i)=r_i\ (i\in O),\qquad
 \tau_O(r_j)=r_{O,\ev_j}\ (j\in X).                             \label{eq:frame-switch-basic}
\end{equation}
Consequently the frame $d_{O,\chi}(\mathcal F_0)$ depends only on the
single sign $\chi(X)$.  If $\chi(X)=\chi'(X)$, then
$d_{O,\chi'}d_{O,\chi}$ stabilizes $\mathcal F_0$ and fixes $O$ pointwise;
on $X\cong AG(4,2)$ it is the translation determined by the character
$\rho=\chi\chi'$ with $\rho(X)=1$.  For $\rho\ne1$ this translation is
nontrivial.
\end{proposition}

\begin{proof}
If $i\in O$, the pairing is $1$, so
Lemma~\ref{lem:nonorthogonal-rigidity} gives $\tau_O(r_i)=r_i$.
For $j\notin O$, the pairing is zero and commutativity gives
\[
       \tau_O(r_j)=r_j*r_O=t_j(r_O)=r_{O,\ev_j}.
\]
Cocode covariance gives
\[
        d_{O,\chi}:d_j\longmapsto d_{O,\chi \ev_j}\qquad(j\in X),
\]
while the eight basic transpositions in $O$ are fixed.

The sixteen evaluation characters $\ev_j$ are exactly the characters of
$B_O$ taking the value $-1$ on $X$.  Hence multiplication by $\chi$ shows
that the image frame depends only on $\chi(X)$.  If $\chi$ and $\chi'$ have
the same parity, put $\rho=\chi\chi'$.  For every $j\in X$ there is a
unique $k\in X$ with
\[
                              \ev_k=\rho \ev_j.
\]
The product $d_{O,\chi'}d_{O,\chi}$ therefore sends $j$ to $k$.  Under the
standard identification $B_O=RM(1,4)$, the characters with value $+1$ on
$X$ are the translation characters: $\ev_{j+v}=\rho_v \ev_j$.  Thus the
induced permutation is $j\mapsto j+v$, nontrivial when $\rho\ne1$.
\end{proof}

We use the following standard $M_{24}$ data, in addition to its
$5$-transitivity: the octad stabilizer induces $A_8$ on the octad,
and its pointwise kernel $T_O^{\mathrm{tr}}\cong2^4$ acts regularly by translations
on the sixteen complementary points; see, for example,
\cite[Chapter~5]{WilsonBook}.  The following lemma will be used for the frame
stabilizer here and for the residue centralizers in Section~\ref{sec:intrinsic-groups}.

\begin{lemma}[Octad translations]
\label{lem:Mathieu-translation-generation}
For $|S|\leq2$, the pointwise stabilizer $(M_{24})_{(S)}$ is generated
by the subgroups $T_O^{\mathrm{tr}}$ for octads $O$ containing $S$.
\end{lemma}
\begin{proof}
Let $M=(M_{24})_{(S)}$ and $R=\langle T_O^{\mathrm{tr}}:S\subseteq O\rangle$.
This subgroup is normal in $M$.  For a set $T\supseteq S$ of at most
four points and distinct $x$, $y\notin T$, the inequality
\[
 \lambda_{|T|}-2\lambda_{|T|+1}>0
\]
provides an octad containing $T$ and avoiding both $x$, $y$.
The five differences for $|T|=0$, $1$, $2$, $3$, $4$ are
$253$, $99$, $35$, $11$, $3$.  A translation in its $T_O^{\mathrm{tr}}$ fixes $T$ and takes
$x$ to $y$.  Consequently $R$ is $(5-|S|)$-transitive on
$I\setminus S$.

Choose a five-set $\Pi\supseteq S$ and let $O$ be its unique octad;
write $M_O=(M_{24})_O$ for its setwise stabilizer.
The transitivity just proved gives $M=R(M_{24})_{(\Pi)}$.
The latter stabilizer is contained in $M_O$ and, modulo $T_O^{\mathrm{tr}}$,
induces the group $A_3$ on $O\setminus\Pi$.  Therefore $M/R$ is a
cyclic group of order dividing three.  On the other hand,
\[
 (M\cap M_O)/T_O^{\mathrm{tr}}\cong A_{8-|S|}
\]
is perfect.  Since $T_O^{\mathrm{tr}}\leq R$ and $M/R$ is abelian, the image of
$M\cap M_O$ in $M/R$ is trivial.  This kills
$(M_{24})_{(\Pi)}$ and proves $M=R$.
\end{proof}

\begin{theorem}[Frame stabilizer]\label{thm:frame-rigidity}
The standard frame has stabilizer
\begin{equation}
 (G_{\rm CP})_{\mathcal F_0}=\overline H\cong2^{12}.M_{24}.
 \label{eq:frame-kernel}
\end{equation}
More generally, the semilinear algebra automorphisms permuting the
basic rays are exactly $ZH$.
\end{theorem}
\begin{proof}
A semilinear algebra automorphism $g$ permuting the basic rays is an
isometry by Corollary~\ref{cor:intrinsic-form}, and permutes $\Rcal$
by Theorem~\ref{thm:extremal-closure}.  It therefore preserves the octadic
supports of Lemma~\ref{lem:intrinsic-support}.  Since the octads span
$\C$ by Lemma~\ref{lem:Witt-numerology}, its coordinate permutation
preserves $\C$ and belongs to $M_{24}=\Aut(\C)$.  Choose $h\in H$
inducing that permutation.
The map $h^{-1}g$ fixes all basic rays, so
Lemma~\ref{lem:basic-ray-scalar-rigidity} gives
$h^{-1}g=\lambda T_\delta$ for some $\lambda\in\mu_3$, $\delta\in\C^*$.
Thus $g\in ZH$; conversely every element of $ZH$ permutes the basic
rays.  This also proves $(G_{\rm CP})_{\mathcal F_0}\leq\overline H$.

The subgroup $(G_{\rm CP})_{\mathcal F_0}$ contains $E_0$.  Its
coordinate image contains every $T_O^{\mathrm{tr}}$ by
Proposition~\ref{prop:intrinsic-frame-switch}.  Lemma~\ref{lem:Mathieu-translation-generation}
with $S=\varnothing$ makes this image all of $M_{24}$.  Since $E_0$
is the kernel of $\overline H\to M_{24}$, it follows that
$(G_{\rm CP})_{\mathcal F_0}=\overline H$.
\end{proof}

\subsection{Frame conjugacy and the group order}\label{sec:frames-order}

\begin{definition}
A \emph{Fischer frame} is a maximal subset of $\Dcal_{\rm CP}$ whose
members commute pairwise.  The standard frame is the set
$\mathcal F_0=\{d_i:i\in I\}$ introduced above.
\end{definition}

\begin{lemma}\label{lem:standard-frame}
The set $\mathcal F_0$ is a Fischer frame of cardinality $24$, and
$\langle\mathcal F_0\rangle=E_0\cong2^{12}$.
\end{lemma}

\begin{proof}
The basic involutions commute and have precisely the Golay-code relations
by Corollary~\ref{cor:relations}.  By
Lemma~\ref{lem:reflection-calculus}(b) and
Lemma~\ref{lem:intrinsic-support}, an octadic transposition commutes with
exactly the eight basic transpositions in its octad support, and a duadic
transposition with exactly the two in its duad support.  Hence no member of
$\Dcal_{\rm CP}\setminus\mathcal F_0$ commutes with all of
$\mathcal F_0$, proving maximality.
\end{proof}

\begin{theorem}[Frame conjugacy]\label{thm:frame-conjugacy}
All Fischer frames are conjugate in $G_{\rm CP}$.  In particular every
frame has $24$ members.
\end{theorem}

\begin{proof}
Let $\mathcal F$ be a frame and let $S$ be a Sylow $2$-subgroup containing
the elementary abelian group $\langle\mathcal F\rangle$.  If
$d$, $e\in\Dcal_{\rm CP}\cap S$, then $de\in S$ and its order is at most $3$.
It cannot have order $3$, so $d$ and $e$ commute.  Thus
$\Dcal_{\rm CP}\cap S$ is a commuting set containing $\mathcal F$, and
maximality gives
\[
                         \mathcal F=\Dcal_{\rm CP}\cap S.
\]
For a second frame choose a Sylow $2$-subgroup $S'$.  Sylow conjugacy sends
$S$ to $S'$, and since $\Dcal_{\rm CP}$ is a conjugacy class it sends
$\Dcal_{\rm CP}\cap S$ to $\Dcal_{\rm CP}\cap S'$.  Hence the two frames
are conjugate.  The standard frame has size $24$.
\end{proof}

\begin{proposition}\label{prop:small-clique}
For $0\leq s\leq5$, the group $G_{\rm CP}$ is transitive on ordered
commuting $s$-tuples of distinct members of $\Dcal_{\rm CP}$.  Hence the
number of common commuting extensions is uniform, and for $0\leq s\leq5$
one has
\begin{equation}
\begin{array}{c|rrrrrr}
 s&0&1&2&3&4&5\\ \hline
 a_s&306936&31671&3510&693&180&51.
\end{array}                                                     \label{eq:clique-numbers}
\end{equation}
\end{proposition}

\begin{proof}
Every commuting set extends to a frame.  Move that frame to
$\mathcal F_0$ by Theorem~\ref{thm:frame-conjugacy}; its stabilizer induces
the full $M_{24}$ by Theorem~\ref{thm:frame-rigidity}.  The
$5$-transitivity of $M_{24}$ gives the first assertion.

For an $s$-subset $S\subseteq\mathcal F_0$, the number of octads containing
$S$ is $\lambda_s$ from Lemma~\ref{lem:Witt-numerology}.  Put
\[
                  (n_0,\ldots,n_5)=(276,23,1,0,0,0)
\]
for the corresponding numbers of duads.  Therefore
\[
                         a_s=(24-s)+32\lambda_s+1024n_s,
\]
which gives \eqref{eq:clique-numbers}.
\end{proof}

\begin{theorem}[Frames through a pentad]\label{thm:pentad-frames}
Every commuting pentad is contained in exactly three Fischer frames.
\end{theorem}

\begin{proof}
By Proposition~\ref{prop:small-clique}, take a coordinate pentad
$\Pi\subseteq\mathcal F_0$.  A transposition commuting with all five members
has coordinate support containing $\Pi$, by
Lemma~\ref{lem:intrinsic-support}.  It is therefore one of the nineteen
remaining basic transpositions or an octadic transposition over the unique
octad $O$ containing $\Pi$; no duadic support contains five points.

If no octadic transposition occurs, maximality forces the standard frame.
If an octadic transposition over $O$ occurs, the only compatible basic
transpositions outside $\Pi$ are the three indexed by $O\setminus\Pi$.
For distinct $\chi$, $\chi'\in B_O^*$, substitution in \eqref{eq:5.1} and
character orthogonality give
\[
 (r_{O,\chi},r_{O,\chi'})
      =\frac12\bigl(1+(\chi\chi')(X)\bigr).
\]
Thus compatible octadic phases have one fixed value on $X=O^c$, so only
one parity class of sixteen can occur.  All three remaining basic
transpositions and all sixteen phases of that parity are mutually
compatible, so maximality forces all of them.  The two parity choices are
precisely the two octadic switches of
Proposition~\ref{prop:intrinsic-frame-switch}.  Together with
$\mathcal F_0$, these are the three frames through $\Pi$.
\end{proof}

\begin{theorem}[Frame count and group order]\label{thm:order-CP}
The number of Fischer frames is
\begin{equation}
 N_{\rm fr}=\frac{3a_0a_1a_2a_3a_4}{24\cdot23\cdot22\cdot21\cdot20}
          =2\,503\,413\,946\,215.                             \label{eq:frame-count}
\end{equation}
Consequently
\begin{align}
 |G_{\rm CP}|&=N_{\rm fr}\,2^{12}|M_{24}|                     \label{eq:group-order}\\
   &=2^{22}3^{16}5^2 7^3\cdot11\cdot13\cdot17\cdot23\cdot29\notag\\
   &=2\,510\,411\,418\,381\,323\,442\,585\,600.\notag
\end{align}
\end{theorem}
\begin{proof}
Count pairs $(\mathcal F,\Pi)$ with $\Pi$ an ordered pentad in the
frame $\mathcal F$.  Proposition~\ref{prop:small-clique} gives
$a_0a_1a_2a_3a_4$ ordered commuting pentads, each in three frames by
Theorem~\ref{thm:pentad-frames}.  Every frame contains
$24\cdot23\cdot22\cdot21\cdot20$ ordered pentads.  This proves
\eqref{eq:frame-count}.

Frames form one orbit and have stabilizer $2^{12}.M_{24}$ by
Theorems~\ref{thm:frame-conjugacy} and~\ref{thm:frame-rigidity}.
Orbit--stabilizer and $|M_{24}|=2^{10}3^3\cdot5\cdot7\cdot11\cdot23$
give \eqref{eq:group-order}.
\end{proof}

\section{Residue orders and elementary simplicity}
\label{sec:intrinsic-groups}

Write $G_{24}=G_{\rm CP}$ and $G_{24}^+=G_{24}'$.  For a subset
$S\subseteq I$ with $s=|S|\leq2$, put
\begin{equation}
 \begin{gathered}
 E_S=\langle d_i:i\in S\rangle,\qquad
 C_S=C_{G_{24}}(d_i:i\in S),\\
 D_S=\{d\in\Dcal_{\rm CP}\setminus\{d_i:i\in S\}:
                  [d,d_i]=1\ (i\in S)\}.
 \end{gathered}
 \label{eq:residue-definitions}
\end{equation}
Thus $E_S\cong2^s$, $E_S\leq Z(C_S)$, and $|D_S|=a_s$.
For $S=\varnothing$ these conventions give $C_S=G_{24}$ and
$D_S=\Dcal_{\rm CP}$.  The two residue groups are defined by
\begin{equation}
 G_{23}=C_{\{i\}}/\langle d_i\rangle,\qquad
 G_{22}=C_{\{i,j\}}/\langle d_i,d_j\rangle\quad(i\ne j).
 \label{eq:intrinsic-residue-groups}
\end{equation}
Small-clique homogeneity makes their isomorphism types independent of
the chosen point or pair.  We prove their orders, faithful
three-transposition actions, and simplicity directly.

\subsection{The faithful residue actions and their orders}

\begin{lemma}\label{lem:residue-kernels}
The kernel of the conjugation action of $C_S$ on $D_S$ is exactly $E_S$.
\end{lemma}
\begin{proof}
For $s=0$ this is the faithful action already constructed.  Suppose
$s=1$ or $2$, and let $g$ be in the kernel.  It fixes the basic
transpositions outside $S$, and centralizes those in $S$, so it fixes
the entire standard frame pointwise.  Theorem~\ref{thm:frame-rigidity}
therefore puts $g$ in the cocode group $E_0$.

For a duad $p$, the cocode kernel on its duadic fibre is
$\C_p^\perp=\langle d_i:i\in p\rangle$:
Proposition~\ref{prop:chosen-duadic-fibre} gives its order four,
and the two coordinate evaluations on $p$ annihilate $\C_p$.
If $s=2$, take $p=S$; all roots in this fibre belong to $D_S$, so
$g\in E_S$.  If $S=\{i\}$, take two distinct duads
$\{i,j\}$ and $\{i,k\}$.  Their kernels intersect in $\langle d_i\rangle$,
since three coordinate cocodes are independent: a nonzero word
supported on three points cannot belong to the Golay code.  This again
gives $g\in E_S$.  Conversely every element of $E_S$ fixes $D_S$.
\end{proof}

\begin{theorem}[Residue orders]\label{thm:residue-orders}
The groups in \eqref{eq:intrinsic-residue-groups} act faithfully on
$31671$ and $3510$ points, respectively, and
\begin{align}
 |G_{24}^+|
  &=2^{21}3^{16}5^2 7^3\cdot11\cdot13\cdot17\cdot23\cdot29,
       \label{eq:three-intrinsic-orders}\\
 |G_{23}|
  &=2^{18}3^{13}5^2\cdot7\cdot11\cdot13\cdot17\cdot23
    =4\,089\,470\,473\,293\,004\,800,\notag\\
 |G_{22}|
  &=2^{17}3^9 5^2\cdot7\cdot11\cdot13
    =64\,561\,751\,654\,400.\notag
\end{align}
\end{theorem}
\begin{proof}
The order of $G_{24}$ is Theorem~\ref{thm:order-CP}, and
Proposition~\ref{prop:parity-perfect} gives the index two.  The
stabilizers of an ordered commuting point and pair have indices
$a_0$ and $a_0a_1$, respectively, by
Proposition~\ref{prop:small-clique}.  Dividing these stabilizer orders
by $|E_S|=2^s$ gives
\[
 |G_{23}|=\frac{|G_{24}|}{2a_0},\qquad
 |G_{22}|=\frac{|G_{24}|}{4a_0a_1}.
\]
Substitution gives \eqref{eq:three-intrinsic-orders}.  The preceding
lemma proves faithfulness on $D_S$, whose sizes are $a_1$, $a_2$.
\end{proof}

\subsection{Generation of the centralizers}

\begin{proposition}
\label{prop:residue-generation}
For $|S|\leq2$,
\[
                 C_S=\langle E_S,D_S\rangle.
\]
Thus $C_S/E_S$ is generated by the images of $D_S$.
\end{proposition}
\begin{proof}
The assertion for $S=\varnothing$ is the definition.  In the other
cases put $L=\langle E_S,D_S\rangle\trianglelefteq C_S$.
It contains $E_0$ and the group generated by every frame through $S$.
The Sylow argument for Theorem~\ref{thm:frame-conjugacy} works inside
$L$: if a Sylow $2$-subgroup $P$ contains a frame $\mathcal F$ through
$S$, then
\[
             D_S\cap P=\mathcal F\setminus\{d_i:i\in S\}.
\]
Indeed, two distinguished involutions in a $2$-group commute, and
maximality of the frame applies.  Sylow conjugacy in $L$ therefore
makes $L$ transitive on frames through $S$.

The intersection $L\cap\overline H$ contains $E_0$.  Its image in
$(M_{24})_{(S)}$ contains $T_O^{\mathrm{tr}}$ whenever $S\subseteq O$: the products
of two octadic reflections of the same phase parity induce exactly
these translations by
Proposition~\ref{prop:intrinsic-frame-switch}, and both reflections
belong to $D_S$.  Lemma~\ref{lem:Mathieu-translation-generation}
then gives $\overline H_{(S)}\leq L$.  Finally, for $g\in C_S$,
frame transitivity in $L$ supplies $a\in L$ such that
$ag\in\overline H_{(S)}$.  Hence $g\in L$.
\end{proof}

\subsection{Rank three and primitivity from the coordinate supports}

\begin{proposition}[Rank-three actions]\label{prop:residue-rank-three}
For $s=0$, $1$, $2$, the faithful group $C_S/E_S$ has rank three on $D_S$.
Its point-stabilizer orbit lengths are
\begin{equation}
\begin{array}{c|r|r|r}
 s&|D_S|&\text{other commuting points}&\text{noncommuting points}\\ \hline
 0&306936&31671&275264\\
 1&31671&3510&28160\\
 2&3510&693&2816.
\end{array}
\label{eq:intrinsic-rank-three}
\end{equation}
All three actions are primitive.
\end{proposition}
\begin{proof}
Small-clique homogeneity makes $C_S$ transitive on $D_S$ and
$C_{S\cup\{x\}}$ transitive on the points of $D_S$ commuting with a
fixed basic point $d_x$, where $x\notin S$.  These sets have sizes
$a_s$ and $a_{s+1}$.

The points of $D_S$ not commuting with $d_x$ have precisely two
coordinate types: octads $F$ with $S\subseteq F$, $x\notin F$,
and duads $p$ with $S\subseteq p$, $x\notin p$, with all their
phases.  The subgroup $\overline H_{(S\cup\{x\})}$ is transitive
on each type.  For octads, this follows because the octad stabilizer
induces $A_8$ on its points and its translation kernel is transitive
on its complement: it is transitive on configurations consisting
of the ordered $s$ points inside an octad and one point outside.
For duads it follows from $5$-transitivity.  The full cocode, which
fixes the basic frame, is transitive on the phases in both cases.

These two orbits are fused by an octadic reflection.  Fix an octad
$F$ containing $S$ and avoiding $x$.  There is an octad $O$ with
\[
 S\cup\{x\}\subseteq O,\qquad |O\cap F|=2.
\]
To see this, enlarge $S$ to a two-set $T\subset F$.
Of the $\lambda_3=21$ octads through $T\cup\{x\}$, exactly
$\binom62=15$ meet $F$ in four points: the two extra points of $F$
complete a five-set, which lies in a unique octad.  The other six
satisfy $O\cap F=T$ and give the required choice.
Every octadic reflection over $O$ fixes $S\cup\{x\}$.
By Lemma~\ref{lem:orthogonal-octadic-pairs} its action on a root
over $F$ is their product, a duadic root over $O\cap F$.
This duad contains $S$ and avoids $x$, so the two noncommuting
orbits fuse.  The remaining orbit size is $a_s-1-a_{s+1}$,
proving the table.  Faithfulness is
Lemma~\ref{lem:residue-kernels}.

A block containing a point must be a union of that point-stabilizer's
orbits.  Besides a singleton and the whole set, its only possible
sizes are one plus either nontrivial subdegree.  For the three rows
these are
\[
 (31672,275265),\qquad (3511,28161),\qquad (694,2817).
\]
The corresponding degrees are $306936$, $31671$, and $3510$;
neither possible proper block size divides its degree.  The size of
a block in a finite transitive action must divide the degree,
so all three actions are primitive.
\end{proof}

\begin{corollary}
\label{cor:residue-classes}
For $s=1$, $2$, the images of $D_S$ in $C_S/E_S$ are distinct
involutions, form one conjugacy class, and generate the quotient.
Products of distinct images have order two or three according as
the original involutions commute or do not commute.
\end{corollary}
\begin{proof}
Each member of $D_S$ moves some other member, so its image has
order two.  Equality of two images defines an invariant equivalence
relation on $D_S$.  Primitivity makes this either equality or the
universal relation.  It is not universal: a product of order three
cannot lie in the $2$-group $E_S$.  Thus the images are distinct.
Transitivity of $C_S$ gives the conjugacy-class assertion, and
Proposition~\ref{prop:residue-generation} gives generation.
The same $2$-group argument preserves product order three;
a product of distinct commuting images is a nonidentity involution.
\end{proof}

\subsection{Simplicity}

\begin{lemma}
\label{lem:primitive-involution-simplicity}
Let $J$ be a finite group generated by a conjugacy class $D$ of
involutions, and suppose conjugation on $D$ is faithful and primitive.
Every nontrivial normal subgroup contains $J'$.  Consequently:
\begin{enumerate}[label=(\alph*)]
\item if $J$ is perfect, then $J$ is simple;
\item if $[J:J']=2$ and $|J'|\ne|D|^2$, then $J'$ is simple.
\end{enumerate}
\end{lemma}
\begin{proof}
The orbits of a normal subgroup form a block system.  A nontrivial
normal subgroup $N$ is therefore transitive, since the action is
faithful.  All members of $D$ then have the same image in $J/N$.
This quotient is generated by one involution, so $J'\leq N$.
This proves the first assertion and (a).

For (b) put $K=J'$ and choose $d\in D\setminus K$; then
$J=K\rtimes\langle d\rangle$.  Suppose that
$1<N<K$ is normal in $K$.  Both $N\cap N^d$ and $NN^d$ are normal
in $J$, so the preceding assertion gives
\[
                 N\cap N^d=1,\qquad NN^d=K.
\]
The two factors commute; hence $K=N\times N^d$, with $d$ interchanging
them.  Its centralizer in $K$ is the diagonal subgroup of order $|N|$.
Also $K$ is transitive on $D$, since $d\in C_J(d)$ and
$J=K\langle d\rangle$.  Thus
\[
 |D|=[K:C_K(d)]=|N|,\qquad |K|=|D|^2,
\]
contrary to the hypothesis.  No proper nontrivial $N$ exists.
\end{proof}

\begin{theorem}[Simplicity]
\label{thm:intrinsic-simplicity}
The groups $G_{24}^+$, $G_{23}$, and $G_{22}$ are nonabelian simple.
\end{theorem}
\begin{proof}
For $G_{24}$, Proposition~\ref{prop:parity-perfect} gives
$[G_{24}:G_{24}^+]=2$, and
Proposition~\ref{prop:residue-rank-three} gives a faithful primitive
action on its generating involution class.  Its subgroup $G_{24}^+$
has order divisible by five, whereas $306936^2$ does not.
Lemma~\ref{lem:primitive-involution-simplicity}(b) proves its simplicity.

For either residue $J=C_S/E_S$ with $s=1$, $2$, the preceding corollary
shows that $J/J'$ is generated by a single involution $u$.
Choose an octad $O$ meeting $S$ in exactly one point.  Such octads
exist: for $s=1$ there are $253$, and for $s=2$ there are
$253-77=176$ through a specified point of $S$ and not the other.
The Golay cocode relation is
\begin{equation}
                       \prod_{i\in O}d_i=1.
                       \label{eq:odd-residue-relation}
\end{equation}
In $J$ the unique factor indexed by $O\cap S$ disappears, leaving
seven members of the distinguished class.  In the abelianization
this gives $u^7=1$.  Together with $u^2=1$ it forces $u=1$;
thus $J$ is perfect.  Primitivity and
Lemma~\ref{lem:primitive-involution-simplicity}(a) prove simplicity.
These orders exclude abelian simple groups.
\end{proof}

\subsection{Central extensions}

\begin{proposition}
\label{prop:intrinsic-centralizer-tower}
For distinct commuting $d$, $e\in\Dcal_{\rm CP}$,
\[
 C_{G_{24}}(d)=\langle d\rangle\times C_{G_{24}^+}(d),
       \qquad C_{G_{24}^+}(d)\cong G_{23}.
\]
The image $\bar e$ of $e$ in $G_{23}$ has centralizer $B$ fitting into
\begin{equation}
       1\longrightarrow\langle\bar e\rangle
       \longrightarrow B\longrightarrow G_{22}\longrightarrow1.
       \label{eq:intrinsic-double-cover}
\end{equation}
The group $B$ is perfect, its center is exactly
$\langle\bar e\rangle$, and this central extension is nonsplit.
\end{proposition}
\begin{proof}
The first splitting follows from parity: $d$ is central in its
centralizer and has odd parity.  Multiplication by $d$ carries its
odd coset to its even subgroup, giving the stated direct product.

By small-clique homogeneity take $d=d_i$ and $e=d_j$.
Injectivity of the residue class shows that an element of
$C_{\{i\}}/\langle d_i\rangle$ centralizes $\bar e$ precisely
when a representative centralizes $d_j$.  Therefore
\[
 B=C_{\{i,j\}}/\langle d_i\rangle,
\]
which gives \eqref{eq:intrinsic-double-cover}.
By Proposition~\ref{prop:residue-generation}, $B$ is generated by
$\bar e$ and the images of $D_{\{i,j\}}$.  Those images are one
conjugacy class, so write $u$ for their common image in $B/B'$ and
$v$ for that of $\bar e$.  An octad through $i$ and not $j$ gives
$u^7=1$, whence $u=1$.  An octad through $j$ and not $i$ then gives
$vu^7=1$, whence $v=1$.  Thus $B$ is perfect.

Its simple nonabelian quotient $G_{22}$ is centerless, so the center
of $B$ is exactly the kernel of order two.  A splitting would give
$B\cong C_2\times G_{22}$, contradicting perfectness.
\end{proof}

\subsection{The full semilinear automorphism group}
\label{sec:automorphisms}

Write $\Gamma=\Aut_{\rm sl}(A,*)$ for all $\E$-linear or
$\E$-conjugate-linear algebra automorphisms.  By
Corollary~\ref{cor:intrinsic-form} and Theorem~\ref{thm:extremal-closure},
each element of $\Gamma$ preserves $\Rcal$.  The scalar-kernel lemma
identifies its faithful ray image with $\overline\Gamma=\Gamma/Z$.

\begin{corollary}
\label{cor:full-semilinear-group}
Every Fischer frame has stabilizer $2^{12}.M_{24}$, and
\[
            \overline\Gamma=G_{\rm CP},\qquad
            \Aut_{\rm sl}(A,*)=\Gamma=\widehat G.
\]
Thus every semilinear algebra automorphism is generated by root
reflections.
\end{corollary}
\begin{proof}
Frame conjugacy transfers the standard stabilizer to every frame.
For $g\in\Gamma$, covariance makes $g(\mathcal F_0)$ a frame.  Choose
$a\in\widehat G$ whose ray action takes it to $\mathcal F_0$.
Theorem~\ref{thm:frame-rigidity} then gives $ag\in ZH$.
The ray image of $H$ lies in $G_{\rm CP}$ by that same theorem, so
$H\leq\widehat G$: for $h\in H$, choose a root word with the same
ray action and apply Lemma~\ref{lem:scalar-kernel}, using
$Z\leq\widehat G$.  Therefore $g\in\widehat G$, proving both equalities.
\end{proof}

\begin{corollary}
\label{cor:triple-cover}
The $\E$-linear subgroup $\widehat G^+$ is perfect and has center
$Z\cong C_3$, with
\[
 1\longrightarrow Z\longrightarrow\widehat G^+
   \longrightarrow G_{24}^+\longrightarrow1
\]
a nonsplit central extension.  Furthermore
\[
 \Aut_{\rm sl}(A,*)=\widehat G
   =\widehat G^+\rtimes\langle\tau_r\rangle,
 \qquad |\widehat G|=3|G_{24}|.
\]
Every root reflection inverts $Z$.
\end{corollary}
\begin{proof}
The exact sequence, perfectness, and nonsplitting were proved in
Proposition~\ref{prop:parity-perfect}.  Simplicity of $G_{24}^+$
now identifies the center of $\widehat G^+$ with $Z$.
A root reflection has order two and odd semilinear parity, so it
splits the full group over $\widehat G^+$.  Conjugate-linearity gives
$\tau_r(\omega I)\tau_r^{-1}=\bar\omega I$.
Corollary~\ref{cor:full-semilinear-group} identifies $\widehat G$
with all semilinear algebra automorphisms.
\end{proof}

\subsection{Identification with Fischer's groups}\label{sec:identification}

It remains to identify the constructed groups.  We use two distinct
uniqueness results:
Fischer's uniqueness of the exceptional three-transposition pairs, and
Aschbacher's comparison with the local characterizations used in the
classification of finite simple groups.

\begin{corollary}[Identification]\label{cor:fischer-identification}
The constructed three-transposition pairs are Fischer's exceptional pairs
$M(24)$, $M(23)$, and $M(22)$.  Consequently
\[
 G_{24}\cong\Fi_{24}'{:}2,\qquad G_{24}^+\cong\Fi_{24}',\qquad
 G_{23}\cong\Fi_{23},\qquad G_{22}\cong\Fi_{22}.
\]
\end{corollary}
\begin{proof}
\emph{Identification of the three-transposition pairs.}
Theorem~\ref{thm:axis-group}, Corollary~\ref{cor:residue-classes},
and Theorem~\ref{thm:intrinsic-simplicity} give centerless
three-transposition pairs with class sizes $306936$, $31671$, and $3510$.
They are almost simple: the residues are simple, while
$C_{G_{24}}(G_{24}^+)$ is normal and intersects the nonabelian simple
index-two subgroup $G_{24}^+$ trivially.  It has order at most two
and would be central if nontrivial, so centerlessness makes it trivial.
Thus conjugation embeds $G_{24}$ in $\Aut(G_{24}^+)$.  In particular,
none of the three groups has a nontrivial solvable normal subgroup.

Their maximal commuting sets have sizes $24$, $23$, and $22$.
Indeed, commutation is preserved in the residue quotients; adjoining
the fixed point or pair to a maximal commuting residue set gives a
frame, of size $24$ by Theorem~\ref{thm:frame-conjugacy}.
Fischer's classification and uniqueness theorem
\cite{FischerThree,FischerWarwick}, with its complete published treatment
in \cite[Part~I, Chapter~5]{AschbacherThree}, now applies.
Comparing these class sizes with the list in
\cite[Theorem~5.1 and Section~6]{HallShpectorov} excludes all classical
and triality cases.  The only symmetric possibility is $S_{784}$,
since $\binom{784}{2}=306936$; its maximal commuting sets have size
$392$, not $24$.  Thus the class and frame sizes select the exceptional
pairs $M(24)$, $M(23)$, and $M(22)$, respectively.

Fischer's theorem gives uniqueness of the exceptional three-transposition
pairs.  A Fischer-space isomorphism gives an isomorphism of the corresponding
centerless groups: in the Fischer space with lines
$\{d,e,ded\}$ for $|de|=3$, conjugation by $d$ fixes $d$ and the points
not collinear with $d$, and interchanges the other two points on each
line through $d$.  A space isomorphism therefore conjugates these point
involutions.  They generate the faithful conjugation action of the
centerless group, so the resulting identification is an isomorphism of
pairs, preserving the distinguished classes.

\emph{Identification with the locally characterized sporadic groups.}
The additional comparison is supplied by Aschbacher's
involution-centralizer characterizations
\cite[Part~II, Introduction and Chapter~11]{AschbacherThree}.
For the two smaller groups, \cite[Theorems~31.1 and~32.1]{AschbacherThree}
identify the locally characterized groups with the three-transposition
groups $M(22)$ and $M(23)$, respectively.
For the largest group, \cite[Theorem~34.1]{AschbacherThree} identifies the
local type $F_{24}$ with $M(24)'$.
In addition, \cite[Theorem~35.1]{AschbacherThree} proves that a group of
type $\Aut(F_{24})$ is a three-transposition group of type $M(24)$;
its uniqueness as such a group then follows from Fischer's theorem.
Thus the standard sporadic names are
$M(22)\cong\Fi_{22}$, $M(23)\cong\Fi_{23}$, and
$M(24)'\cong\Fi_{24}'$, with $M(24)\cong\Fi_{24}'{:}2$.
For this notation comparison, see also
\cite[Section~6]{AschbacherGuralnickSegev}.
Together with the pair identifications above, these give the asserted
isomorphisms.
\end{proof}

\begin{remark}[Construction and recognition]\label{rem:logical-direction}
The classification and uniqueness results cited here identify the named
pairs and relate them to the local characterizations of the sporadic
groups.  They supply none of our construction, order, simplicity, or
central-extension proofs.  In particular, the use of Aschbacher's
Part~II characterizations is distinct from importing its Monster-based
existence construction.  The Schur multiplier is not determined here.
\end{remark}

\section{The Eisenstein root algebra and its reductions}
\label{sec:integral}

The root configuration defines an integral Eisenstein algebra.  Its
odd-prime localizations recover the coordinate algebra, and reduction
modulo $1-\omega$ gives a ternary algebra with a faithful group action.
The moment identities give the discriminant bound and the ternary
commuting-matrix identity.

\subsection{Quadratic reconstruction and moments}
\label{sec:moments}

The symmetric-square basis determines both the product and the moments.
The second moment will bound the integral discriminant and give the
square-zero identity for the ternary commuting matrix.

\begin{theorem}[Quadratic product rigidity]
\label{thm:quadratic-product-rigidity}
The symmetric squares of one root representative from each ray form a
basis of $\operatorname{Sym}^2_\E(A)$.  There is at most one commutative
conjugate-bilinear product on $A$ satisfying $r*r=10r$ on these roots.
\end{theorem}
\begin{proof}
Lemma~\ref{lem:quadratic-unisolvence} gives independence, and
$|\Rcal|=\binom{784}{2}$ gives a basis.  Independence over $\mathbb C$
implies independence over $\E$.  A commutative conjugate-bilinear
product is a linear map $\operatorname{Sym}^2_\E(\overline A)\to A$,
whose values on the conjugate symmetric-square basis are prescribed.
\end{proof}

Choose any representative $r$ of each ray and define
\begin{align}
 M_2(x,y)&=\sum_{[r]\in\Rcal}(x,r)(r,y),\label{eq:8.0}\\
 M_3(x,y,z)&=\sum_{[r]\in\Rcal}(x,r)(y,r)(z,r).
                                                   \label{eq:8.1}
\end{align}
Both are independent of the choices, since $\lambda\bar\lambda=1$
and $\lambda^3=1$ for $\lambda\in\mu_3$.

\begin{theorem}[Moment identities]
\label{thm:second-moment}
The second moment, product, and third moment are
\begin{align}
 M_2(x,y)&=3528(x,y),                                      \label{eq:8.0a}\\
 x*y&=\frac1{360}\sum_{[r]\in\Rcal}(r,x)(r,y)r,
                                            \label{eq:product-reconstruction}\\
 M_3(x,y,z)&=360\Phi(x,y,z).                               \label{eq:8.2}
\end{align}
\end{theorem}

\begin{proof}
Put
\[
 Sx=\sum_{[r]}(x,r)r,\qquad
 B(x,y)=\sum_{[r]}(r,x)(r,y)r.
\]
The second-moment operator $S$ is $\E$-linear and self-adjoint.  At a root
$s$, the pairing identity gives
\[
 B(s,s)=\sum_{[r]}(s,r)r+(9^2-9)s=(S+72I)s.
\]
Since the conjugate symmetric squares are a basis,
\begin{equation}
                B(x,y)=\frac1{10}(S+72I)(x*y).              \label{eq:frame-centroid-interpolation}
\end{equation}
The cubic $(x,B(y,z))=M_3(x,y,z)$ is symmetric.  Self-adjointness of
$S$ and symmetry of $\Phi$ therefore give
\[
             \Phi(Sx,y,z)=\Phi(x,Sy,z)=\Phi(x,y,Sz),
\]
and nondegeneracy of the form implies
\begin{equation}
                S(x*y)=(Sx)*y=x*(Sy).                      \label{eq:frame-centroid}
\end{equation}
Thus $S$ belongs to the self-adjoint centroid of the algebra.

One reflecting root forces this centroid element to be scalar.
Fix such a root $r$ and let $L_r(x)=x*r$.  By
\eqref{eq:root-R-square},
\[
             L_r^2x=x+11(x,r)r;
\]
in particular $L_r$ is invertible.  Equation
\eqref{eq:frame-centroid} gives $L_r(Sr)=10Sr$, so
$99Sr=11(Sr,r)r$.  Hence $Sr=\lambda r$ with $\lambda\in\mathbb R$.
Again by \eqref{eq:frame-centroid}, $SL_r=\lambda L_r$;
invertibility gives $S=\lambda I$.  Finally
\begin{equation}
       783\lambda=\operatorname{tr}S=9|\Rcal|
                        =9\cdot306936,
       \qquad\lambda=3528.                                \label{eq:moment-trace-check}
\end{equation}
This proves the second moment.  Substitution in
\eqref{eq:frame-centroid-interpolation} yields $B=360*$, proving
both remaining identities.
\end{proof}

\begin{remark}
The Gram matrix $G$ of the root representatives has eigenvalues
$3528$ and $0$; the matrix $72I+\overline G$ has eigenvalues
$3600$ and $72$.
\end{remark}

\begin{corollary}
\label{thm:ray-preservation-criterion}
Let $g:A\to A$
be $\E$-linear or $\E$-conjugate-linear.  If $g$ permutes the
three-element rays in $\Rcal$, then it is respectively unitary or
antiunitary, and it is an algebra automorphism.
\end{corollary}

\begin{proof}
The roots span $A$, so $g$ is invertible.  Transporting the product by
$g$ gives the commutative conjugate-bilinear product
$x\star y=g^{-1}(g(x)*g(y))$, also when $g$ is conjugate-linear.
Since $g$ maps roots to roots, $r\star r=10r$; quadratic product
rigidity gives $\star=*$.  Thus $g$ is an algebra automorphism,
and Corollary~\ref{cor:intrinsic-form} gives the isometry assertion.
\end{proof}

\subsection{The Eisenstein root module}
\label{sec:Eisenstein-lattice}

Write $\mathfrak o=\mathbb Z[\omega]$ and choose
one signed coordinate $x_D$ over each octad.  Put
\begin{equation}
 L_{\rm CP}=\sum_{[r]\in\Rcal}\mathfrak o r,
 \qquad
 M=\bigoplus_{i\in I}\mathfrak o u_i
       \oplus\bigoplus_{D\in\Ocal}\mathfrak o x_D.
 \label{eq:CP-Eisenstein-lattice}
\end{equation}
The first module is independent of the chosen root representatives.
Let $R=\mathfrak o[1/2]$ and write $L_R=R\otimes_{\mathfrak o}L_{\rm CP}$.

\begin{proposition}
\label{prop:CP-dyadic-localization}
$L_{\rm CP}$ is a free
$\mathfrak o$-module of rank $783$, and
\begin{equation}
                         L_R=M_R.
 \label{eq:CP-dyadic-equality}
\end{equation}
The latter module is closed under $*$ and is self-dual for the Hermitian
form over $R$.  In particular $L_{\rm CP}$ is self-dual at every odd
rational prime.
\end{proposition}

\begin{proof}
The chosen root representatives have coordinates in $\frac18M$, so
$L_{\rm CP}\subseteq\frac18M$.  Conversely, the coordinate extraction formulas in
Lemma~\ref{lem:root-spanning} express every $u_i$ and $x_D$ as a
linear combination of basic and octadic roots over $R$.
This proves \eqref{eq:CP-dyadic-equality}.  The finite
set of generators and torsion-freeness show that $L_{\rm CP}$ is a free
module over the Euclidean domain $\mathfrak o$, and the coordinate
expressions give its rank $783$.

Every structure constant in \eqref{eq:2.5}--\eqref{eq:2.9} belongs to
$R$, so $M_R*M_R\subseteq M_R$.  Its coordinate Gram matrix is
\[
                  \operatorname{diag}(\tfrac18 I_{24},I_{759}),
\]
whose determinant $8^{-24}$ is a unit of $R$.  This proves self-duality
over $R$, and hence after localization at every odd prime.
\end{proof}

\begin{proposition}
\label{prop:CP-integral-order}
The root module $L_{\rm CP}$ is an integral Hermitian lattice and is closed under $*$.
Its Hermitian discriminant module $L_{\rm CP}^{\#}/L_{\rm CP}$ is
$2$-primary and is annihilated by $8$.
\end{proposition}

\begin{proof}
Lemma~\ref{lem:reflecting-root-pairing} gives
$(L_{\rm CP},L_{\rm CP})\subseteq\mathfrak o$.  The root products are
\[
 r*r=10r,\qquad
 [r*s]\in\Rcal\text{ if }(r,s)=0,
 \qquad
 r*s=cr+\bar c\,s\text{ if }c=(r,s)\in\mu_3,
\]
by Lemma~\ref{lem:reflection-calculus}.  Conjugate-bilinearity proves closure over
$\mathfrak o$.

Proposition~\ref{prop:CP-dyadic-localization} makes the finite
discriminant module $2$-primary.  For $x\in L_{\rm CP}^{\#}$ the
second-moment identity gives
\[
                    3528x=\sum_{[r]}(x,r)r\in L_{\rm CP}.
\]
Since $3528=8\cdot441$ and multiplication by $441$ is invertible on a
finite $2$-primary module, multiplication by $8$ annihilates it.
\end{proof}

\subsection{The ternary algebra and its orthogonal quotient}

\begin{proposition}
\label{prop:CP-mod-three}
Let $\mathfrak p=(1-\omega)$.  Then
\[
 \overline A=L_{\rm CP}/\mathfrak pL_{\rm CP}
           \cong M_R/\mathfrak pM_R
\]
is a $783$-dimensional commutative algebra over $\F_3$ with a
nondegenerate symmetric invariant form.  The product is defined via the
localized algebra $M_R$; by Proposition~\ref{prop:CP-integral-order},
it is also the ordinary reduction of the integral order.  The reductions
of the $306936$ rays in $\Rcal$ are
distinct nonzero isotropic idempotents.

The induced action of
$\widehat G$ factors through a faithful action of $G_{\rm CP}$ on
$\overline A$.
\end{proposition}

\begin{proof}
The integer $2$ is invertible modulo $\mathfrak p$, so localization does
not change the quotient.  Proposition~\ref{prop:CP-dyadic-localization}
gives the displayed isomorphism.  Conjugation becomes trivial modulo
$\mathfrak p$, so the product becomes bilinear and the Hermitian form
becomes symmetric.  Its coordinate Gram matrix reduces to
$\operatorname{diag}(2I_{24},I_{759})$, which is nondegenerate.  Cubic
symmetry from Proposition~\ref{prop:cubic} gives invariance.  The root
equations reduce to $\bar r^{\,2}=\bar r$ and
$(\bar r,\bar r)=0$.

The coordinate profiles prove injectivity of reduction.
The reduced $U$-coordinates of a basic root are $2$ at one point and $1$
at the other $23$ points.  Those of an octadic or duadic root are $1$ on
its support and $2$ off it.  These profiles distinguish all three types
and their supports, and also show that every reduction is nonzero.

Within an octadic fibre, the thirty affine-hyperplane coordinates have
nonzero coefficients proportional to $\chi(b)$ after reduction.  These
words generate $B_F$ (a complementary hyperplane pair sums to $I\setminus
F$), and $1\ne-1$ in $\F_3$; hence they recover $\chi$.  Within a duadic
fibre, the surviving weight-$8$ coordinates recover the values of $\psi$
on all octads disjoint from $p$.  These octads span $\C_p$: choose
$F\cap F'=p$ and use $\C_p=B_F\oplus B_{F'}$ from \eqref{eq:6.2}, with
each summand generated by its octads.  Thus they recover $\psi$ as well.
This proves that all $306936$ reductions are distinct.

Global root closure makes $\widehat G$
preserve $L_{\rm CP}$.  All scalar cube roots reduce to $1$, so its
action factors through $G_{\rm CP}$.  The reduced pairing is $1$ for
distinct nonorthogonal rays and $0$ otherwise.

An element acting trivially on $\overline A$ fixes the reductions of all
rays and therefore every ray.  Since $G_{\rm CP}$ acts faithfully on
$\Rcal$, that element is the identity.
\end{proof}

\begin{corollary}
\label{cor:CP-ternary-graph-rank}
Let $B$ be the
adjacency matrix of the commuting codiagram on $\Rcal$.  Then
\[
                  \operatorname{rank}_{\F_3}B=783,
                  \qquad B^2=0\quad\text{over }\F_3.
\]
Consequently its row space is a self-orthogonal ternary code of length
$306936$ and dimension $783$.
\end{corollary}

\begin{proof}
Let $V$ be the $783\times306936$ matrix whose columns are the reduced
roots, in any basis of $\overline A$, and let $J$ be the nonsingular
matrix of its symmetric form.  The columns span $\overline A$ because
the roots generate $L_{\rm CP}$.  The preceding proposition gives
$B=V^tJV$.  The map $V$ is onto and $V^tJ$ is injective; thus
$\operatorname{rank}B=783$.  Reduction of the second-moment identity \eqref{eq:8.0a} gives
$VV^tJ=3528I=0$, and hence
$B^2=V^tJ(VV^tJ)V=0$.  Since $B$ is symmetric, its row space is
self-orthogonal.
\end{proof}

\begin{corollary}
\label{cor:CP-mod-three-quotient}
Put
\[
               w=2\sum_{i\in I}\overline{u_i}\in\overline A.
\]
Then $(w,w)=0$ and $(w,\bar r)=1$ for every root representative $r$
with $[r]\in\Rcal$.  In particular
$w$ is a nonzero vector fixed by $G_{\rm CP}$, and
\[
  0<\F_3w<w^\perp<\overline A,
  \qquad \dim w^\perp=782,
  \qquad Q=w^\perp/\F_3w
\]
is an invariant flag with $Q$ a nondegenerate $781$-dimensional orthogonal
module on which $G_{\rm CP}$ acts faithfully.  Here $Q$ is the quotient
as a module with a form.
\end{corollary}

\begin{proof}
Write $S=\sum_i u_i$ in characteristic zero.  The coordinate formulas give
\[
 (2S,2S)=12,
 \qquad (2S,r_i)=4,
 \qquad (2S,r_{F,\chi})=1,
 \qquad (2S,r_{p,\psi})=\tfrac14.
\]
Reduction gives the stated values.  Every group element preserves the form
and permutes the reduced roots.  Since those roots span $\overline A$,
these pairings determine $w$ uniquely, so the group fixes it.  The radical
of the restricted form on $w^\perp$ is exactly $\F_3w$, proving the
nondegeneracy of $Q$.

The kernel of the action of the isometry stabilizer of the vector $w$
on $Q$ is a $3$-group: in a basis adapted to the flag, its matrices are
upper unitriangular.  The faithfulness on $\overline A$ proved above
therefore makes the kernel of $G_{\rm CP}\to O(Q)$ a normal $3$-subgroup
of $G_{\rm CP}$.  Its intersection with the nonabelian simple subgroup
$G_{24}^+$ is trivial by Theorem~\ref{thm:intrinsic-simplicity}, since
$G_{24}^+$ is not a $3$-group.  It therefore embeds in
$G_{24}/G_{24}^+\cong C_2$, so the kernel is trivial.
\end{proof}

\begin{remark}
The exact $2$-primary discriminant module of $L_{\rm CP}$ and the
irreducibility of the quotient $Q$ are not determined here.  The relation
with integral Moonshine forms and Tate-cohomological constructions is
left for further study.
\end{remark}

\appendix

\section{A reproducible Parker-loop factor set}\label{app:factor-set}

We verify that the ordered-basis factor set in
Lemma~\ref{lem:explicit-factor-set} realizes the square, commutator,
and associator of the Parker loop.  It also gives a direct way to verify
the interchange identity used in Appendix~\ref{app:tensor-contractions}.

\begin{proof}[Proof of Lemma~\ref{lem:explicit-factor-set}]
The bilinear part $\beta$ has zero cocycle defect.  Put
\[
                         \gamma(a,b)=\Theta(a,b,b).
\]
Trilinearity gives
\[
 \gamma(a,b)+\gamma(a+b,c)+\gamma(b,c)+\gamma(a,b+c)
       =\Theta(a,b,c)+\Theta(a,c,b).
\]
With the triangular definition of $\Theta$, the right side differs from
$\iota(a,b,c)$ only by diagonal terms
$\iota(g_i,g_j,g_j)=|g_i\cap g_j|$, which vanish by self-orthogonality.
This proves the associator identity.

For the diagonal identity, polarizing
$q(a)=|a|/4\pmod2$ twice gives its linear, quadratic, and cubic parts.
The term $\Theta(a,a,a)$ is exactly the cubic part, while $\beta(a,a)$ supplies
the linear and quadratic parts; hence $f(a,a)=q(a)$.

Finally let
\[
 G(a,b)=f(a,b)+f(b,a)+\frac{|a\cap b|}{2}.
\]
Put $c(a,b)=f(a,b)+f(b,a)$.  Summing the associator identities
for $(a,a',b)$, $(a,b,a')$, and $(b,a,a')$ gives
\[
             c(a+a',b)+c(a,b)+c(a',b)=\iota(a,a',b).
\]
The function $|a\cap b|/2$ has the same additivity defect.  Thus $G$
is additive in its first variable, and by symmetry in its second.
It therefore suffices to evaluate it on basis vectors, where
$\Theta(g_i,g_j,g_j)=0$ and the triangular definition of $\beta$ gives
$G(g_i,g_j)=0$.  Thus $G=0$, which is the commutator identity.
\end{proof}

\section{The seven cubic-tensor contractions}\label{app:tensor-contractions}

This appendix proves $K=1002C$ from the multiplication table, using
the Witt-design and Parker-loop identities established or stated above.
Throughout, all repeated tensor indices are ordered indices.

\subsection{The point-index contractions}
Put $\gamma=\sqrt2/64$, $\beta=\sqrt2/8=8\gamma$, and
\[
 a_D(i)=4[i\in D]-1,\qquad
 f_{ijk}=-1+16(\delta_{ij}+\delta_{jk}+\delta_{ki})
                        -128\delta_{ij}\delta_{jk}.
\]
The matrices of multiplication by the orthonormal point and octad
vectors, before conjugating the input, have the block forms
\begin{equation}
\begin{gathered}
 R_i=\begin{pmatrix}A_i&0\\0&D_i\end{pmatrix},\qquad
 R_D=\begin{pmatrix}0&v_De_D^t\\e_Dv_D^t&B_D\end{pmatrix},\\
 A_i=(\gamma f_{ijk})_{jk},\qquad
 D_i=\beta\operatorname{diag}(a_E(i)),\qquad v_D=\beta a_D.
\end{gathered}
                                                               \label{eq:contraction-blocks}
\end{equation}
For distinct octads $D$, $E$, the sole nonzero entry in column $E$ of
$B_D$, when it exists, is $\sigma/2$ or $\vth\sigma/2$ in the
coordinate $D+E$ or $\one+D+E$, respectively.  It exists exactly when
$|D\cap E|=4$ or $0$.

We calculate $\overline K$; its coefficient is the contraction of two
copies of $C$ with the three matrices $R_p$, $R_q$, $R_r$ inserted between
them, using the multiplication convention above.

For three point indices put
$\Delta=\delta_{ij}+\delta_{jk}+\delta_{ki}$ and
$\Delta_3=\delta_{ij}\delta_{jk}$.  The four contributions, divided by
$\gamma$, are
\begin{equation}
\begin{array}{c|l}
 \text{indices of the two cubic tensors}&\overline K_{ijk}/\gamma\\ \hline
 UUU&-137/16+1952\Delta-15616\Delta_3\\
 UWW&-1485/16+2640\Delta\\
 WWW\text{, sextet}&-665/2+7840\Delta-89600\Delta_3\\
 WWW\text{, trio}&-4545/8+3600\Delta-23040\Delta_3
\end{array}                                                     \label{eq:point-contraction-table}
\end{equation}
Here the $UWW$ row includes the three positions of the point index.
Their sum is $1002(-1+16\Delta-128\Delta_3)$.

The point block is
\[
 A_i=\gamma\bigl(-J+16(e_i\mathbf1^t+\mathbf1e_i^t+I)
                                      -128e_ie_i^t\bigr).
\]
Expanding the three point blocks in the contraction gives
$\gamma(122f_{ijk}+1815/16)$, the first row of
\eqref{eq:point-contraction-table}; this uses only
$\mathbf1^t\mathbf1=24$ and $e_i^te_j=\delta_{ij}$.
For the second row,
\[
 a_D^tA_i a_D=960\gamma,\qquad
 \sum_D a_D(i)a_D(j)=2816\delta_{ij}-33.
\]
Thus its value is
$960\gamma\beta^4\sum_D(a_D(i)a_D(j)+a_D(j)a_D(k)+a_D(k)a_D(i))$.

For the last two rows, an ordered octad triangle $(D,E,F)$ has
\begin{equation}
\begin{array}{c|r|r|r}
 &\#(D,E,F)&\sum_\ell a_D(\ell)a_E(\ell)
                &\sum_\ell a_D(\ell)a_E(\ell)a_F(\ell)\\ \hline
 \text{sextet}&759\cdot280&24&-120\\
 \text{trio}&759\cdot30&-40&72
\end{array}                                                     \label{eq:triangle-incidence-moments}
\end{equation}
In both rows $\sum_\ell a_D(\ell)=8$.  If the last two entries of a
row are denoted by $s_2$, $s_3$, averaging over point indices gives
\[
 \frac{s_3}{24},\qquad
 \frac{8s_2-s_3}{24\cdot23},\qquad
 \frac{8^3-24s_2+2s_3}{24\cdot23\cdot22}
\]
for three equal, exactly two equal, and three distinct indices.
Multiply by $\beta^3/4$ in the sextet case and by $-3\beta^3/4$ in
the trio case.  This proves the last two rows.  The averaging is valid
because $M_{24}$ is three-transitive on the coordinate positions; it
also follows directly by counting ordered distinct positions in the
fixed tetrad or octad partitions.

\subsection{One point and two equal octad indices}
For $(i,D,D)$ the contributions according to the first cubic tensor's
indices are
\begin{equation}
\begin{array}{c|rrr|r}
 &UUU&UWW&WWW&\overline K_{iDD}/C_{iDD}\\ \hline
 i\in D&65/16&3077/16&6445/8&1002\\
 i\notin D&45/16&1377/16&7305/8&1002
\end{array}                                                     \label{eq:point-octad-contraction-table}
\end{equation}
Write $a=a_D(i)\in\{3,-1\}$.  The $UUU$ contribution is
$15(1+4a)/(16a)$.  Interchanging the two cubic tensors gives the same
term inside $UWW$.  The other $UWW$ terms are
\[
 \frac{740a+645}{8a},\qquad
 \begin{cases}215/4,&a=3,\\225/4,&a=-1,\end{cases}
 \qquad\frac{121}{8}.
\]
The last term comes from the two occurrences of
$(v_D^tv_D)^2=(11/4)^2$.
These expressions follow from \eqref{eq:contraction-blocks},
$v_D^tv_E=3/4$ or $-5/4$ for intersections four or zero, and the
point-refined Witt counts in Lemma~\ref{lem:Witt-numerology}.

The $WWW$ terms containing an index $D$ contribute $215/4$ or $225/4$,
respectively.  For the remaining terms let $F$, $G$ be octads meeting $D$
in zero or four points, let $E$ be the third octad on their triangle,
and put $t=|D\cap E|$.  The products of the four Parker signs reduce to
$(-1)^{|D\cap F\cap G|}$.  The resulting signed weights, after removal
of the common factor $1/16$, sum as follows:
\begin{equation}
\begin{array}{c|rrrr}
 t&8&4&2&0\\ \hline
 \text{signed weight}&550&30240&-17920&-4200.
\end{array}                                                     \label{eq:quadrilateral-count}
\end{equation}
To obtain this row, fix $F$ with $|D\cap F|=4$.  Among the octads $G$
meeting both $D$ and $F$ in four points, the numbers with
$|D\cap F\cap G|=0$, $2$, $3$, $4$ are
\begin{equation}
                              1,\quad72,\quad64,\quad3.       \label{eq:sextet-common-neighbours}
\end{equation}
Indeed the associated sextet gives, respectively, one union of two
tetrads, three choices of four tetrads each supporting $24$ octads,
$64$ octads of pattern $(3,1,1,1,1,1)$, and three unions of two
residual tetrads.  The other configurations use the $30$ hyperplanes
on an octad complement: for a disjoint pair there are $28$ hyperplanes
meeting each in four points.  The contributions to
\eqref{eq:quadrilateral-count} are
\[
\begin{array}{c|r|r|r}
 \text{configuration}&\text{number}&\text{weight}&t\\ \hline
 |D\cap F|=|D\cap G|=|F\cap G|=4,
       \ |D\cap F\cap G|=0&280&1&8\\
 \text{same, triple intersection }2&280\cdot72&1&4\\
 \text{same, triple intersection }3&280\cdot64&-1&2\\
 \text{same, triple intersection }4&280\cdot3&1&0\\
 |D\cap F|=|D\cap G|=4,\ F\cap G=\varnothing&840&-3&0\\
 \text{exactly one of }D\cap F,\,D\cap G\text{ empty}&3360&3&4\\
 D\cap F=D\cap G=\varnothing,\ |F\cap G|=4&840&-3&0\\
 D,\,F,\,G\text{ disjoint}&30&9&8
\end{array}
\]
For a point inside $D$, the average of $a_E(i)$ is $t/2-1$; for a
point outside, it is $1-t/4$.  Hence the weighted sums are $36090$ and
$-13710$.  Dividing by $16a$ gives $6015/8$ and $6855/8$.
Adding the terms containing $D$ proves the $WWW$ entries of
\eqref{eq:point-octad-contraction-table}.

\subsection{Three octad indices}
Let $(D,E,F)$ be a sextet triangle or a trio.  A point index in one
of the two contracted cubic tensors must be sent to an octad index in
the other by \eqref{eq:contraction-blocks}.  Write the contribution,
divided by $\overline{C_{DEF}}$, as
\begin{equation}
                           2A+B+2C_0+D_0.                    \label{eq:WWW-decomposition}
\end{equation}
Here $A$ is the $UUU$--$WWW$ term, $B$ the $UWW$--$UWW$ term,
$C_0$ the $UWW$--$WWW$ term, and $D_0$ the $WWW$--$WWW$ term.
The factors two interchange the two cubic tensors.  Their values are
\begin{equation}
\begin{array}{c|rrrr|r}
 &16A&16B&16C_0&16D_0&2A+B+2C_0+D_0\\ \hline
 \text{sextet}&47&54&1125&13634&1002\\
 \text{trio}&49&150&1023&13738&1002.
\end{array}                                                     \label{eq:WWW-contraction-table}
\end{equation}
The first entry follows by evaluating
\[
 \Phi_U(v_D,v_E,v_F)
 =\gamma\beta^3\left(-8^3
          +16\cdot8\sum_{\{L,M\}}a_L^ta_M
          -128\sum_i a_D(i)a_E(i)a_F(i)\right)
\]
with \eqref{eq:triangle-incidence-moments}; the trio has an additional
minus sign from $C_{DEF}/\overline{C_{DEF}}=-1$.
The second is $6(v_D^tv_E)^2=54/16$ or $150/16$.
For the third entry, fix the point position opposite $D$ and sum
over an octad $G$ meeting each of $E$, $F$ in zero or four points.  Put
$t=|D\cap G|$.  Summing the point index gives
\[
 v_D^tv_G=\frac{2t-5}{4}.
\]
The remaining three octad coefficients, divided by
$\overline{C_{DEF}}$, contribute $q/4$, where the four Parker signs
reduce to $(-1)^{|E\cap F\cap G|}$ and the powers of $\vth$ determine
$q$.  The resulting contributions to sixteen times this single-point
position are
\begin{equation}
\begin{array}{c|c|r|r|r|r}
 &\text{position of }G&\#G&t&q&\#G(2t-5)q\\ \hline
 \text{sextet}&|E\cap G|=|F\cap G|=4,\ u=0&1&8&1&11\\
 &\text{same},\ u=2&72&4&1&216\\
 &\text{same},\ u=3&64&2&-1&64\\
 &\text{same},\ u=4&3&0&1&-15\\
 &\text{one intersection }0,\ \text{one }4&6&4&3&54\\
 &\text{both intersections }0&3&0&-3&45\\ \hline
 \text{trio}&\text{both intersections }4&28&0&-1&140\\
 &\text{one intersection }0,\ \text{one }4&56&4&1&168\\
 &\text{both intersections }0&1&8&3&33
\end{array}                                                    \label{eq:mixed-WWW-count}
\end{equation}
Here $u=|E\cap F\cap G|$ in the sextet rows.  Their multiplicities
are \eqref{eq:sextet-common-neighbours} and the affine-hyperplane
counts on an octad complement.  The sums are $375$ and $341$;
the three point positions give $16C_0=1125$ and $1023$.
The last column is the signed octad count proved in the next subsection.
Together the three coefficient tables prove $K=1002C$.

\subsection{The remaining signed octad count}
For octads $L$, $M$ with intersection four or zero define
\[
 L\diamond M=L+M+\delta(L,M)\one,\qquad
 \delta(L,M)=\begin{cases}0,&|L\cap M|=4,\\1,&L\cap M=\varnothing.
 \end{cases}
\]
Otherwise the pair is inadmissible.  Fix $F=D\diamond E$, and put
$\epsilon=\delta(D,E)$.  For an admissible $G$ relative to $D$, range
through those octads $H$ for which
\[
 J=G\diamond H,\quad G'=D\diamond G,\quad
 H'=E\diamond H,\quad J'=F\diamond J
\]
all exist.  Then $G'$, $H'$, $J'$ automatically form an octad triangle.
Write
\[
\begin{gathered}
 d_0=\delta(G,H),\qquad d_1=\delta(D,G),\qquad d_2=\delta(E,H),\\
 d_3=\delta(F,J),\qquad d_4=\delta(G',H').
\end{gathered}
\]
The contribution to $16D_0$ is
\begin{equation}
\begin{gathered}
 w(G,H)=(-1)^{s+d_0+d_4}
          (-3)^{(d_0+d_1+d_2+d_3+d_4-\epsilon)/2},\\
 s=\frac{|E\cap G|}{2}+|D\cap E\cap H|+|D\cap G\cap H|.
\end{gathered}
                                                               \label{eq:signed-octad-summand}
\end{equation}
The exponent is an integer in $\{0,1,2\}$.  Thus the possible nonzero
weights are among $\pm1$, $\pm3$, $\pm9$.

In a Parker section calibrated on the
central element $\Omega$, the product of the six triangle signs in
this tensor contraction has exponent
\[
\begin{aligned}
 &f(D,E)+f(G,H)+f(D,G)+f(E,H)\\
 &\hspace{13mm}+f(D+G,E+H)+f(D+E,G+H)\\
 &\qquad=\frac{|E\cap G|}{2}
                  +|D\cap E\cap H|+|D\cap G\cap H|\pmod2.
\end{aligned}
\]
This is the interchange identity obtained from the Parker commutator
and associator; alternatively substitution in \eqref{eq:1.12} gives it
by polarization.  Each change of section occurs twice and cancels.
The remaining factors are the displayed powers of $\vth$: the two
$C$'s contribute $(-1)^{d_0+d_4}$, the three multiplication matrices
contribute no additional conjugation, and division by
$\overline{C_{DEF}}$ removes $\vth^\epsilon/2$.
Since $\vth^2=-3$, this proves \eqref{eq:signed-octad-summand}.

To evaluate the sum, use a sextet with tetrads $T_1$, $\ldots$, $T_6$.
For a sextet triangle take
\[
 D=T_1\cup T_2,\quad E=T_1\cup T_3,\quad F=T_2\cup T_3,
\]
and for a trio take the pairs $(T_1,T_2)$, $(T_3,T_4)$, $(T_5,T_6)$.
For fixed $G$, let
\[
                     P_G(z)=\sum_{H\text{ admissible}}z^{w(G,H)}.
\]
Its derivative at $z=1$ is the required signed sum over $H$.
The following two small tables include every $G$ admissible relative
to $D$.  In the first, $R$, $R'$ denote distinct residual tetrads among
$T_4$, $T_5$, $T_6$; a pattern lists the six intersection sizes with the
sextet, and permutations indicated in the description are included in
the multiplicity.
\begin{equation}
\begin{array}{l|r|l|r}
 G\text{ in the sextet case}&\#G&P_G(z)&P_G'(1)\\ \hline
 E\text{ or }F&2&140z+6z^3+3z^9&185\\
 T_1\cup R\text{ or }T_2\cup R&6&4z^{-3}+26z+6z^3&32\\
 T_3\cup R\text{ or }R\cup R'&6&2z^{-9}+4z^{-3}+28z^3+2z^9&72\\
 (0,0,2,2,2,2)&24&40z^3&120\\
 (2,2,2,2,0,0)&72&32z^{-1}+48z+4z^3&28\\
 (2,2,0,2,2,0)&72&16z^{-1}+32z+8z^3&40\\
 (3,1,1,1,1,1),\ 3\text{ in }T_1\text{ or }T_2
       &128&18z^{-1}+38z+6z^3&38
\end{array}                                                     \label{eq:signed-sextet-table}
\end{equation}
For a trio the displayed pattern is $(|D\cap G|,|E\cap G|,|F\cap G|)$:
\begin{equation}
\begin{array}{l|r|l|r}
 G\text{ in the trio case}&\#G&P_G(z)&P_G'(1)\\ \hline
 E\text{ or }F&2&28z+56z^3+z^9&205\\
 (0,4,4)&28&4z^{-3}+90z+4z^3+2z^9&108\\
 (4,4,0)\text{ or }(4,0,4)&56&2z^{-3}+4z^{-1}+28z+2z^3&24\\
 (4,2,2)&224&40z&40
\end{array}                                                     \label{eq:signed-trio-table}
\end{equation}
Consequently
\begin{align*}
 16D_0&=2\cdot185+6\cdot32+6\cdot72+24\cdot120\\
      &\qquad+72\cdot28+72\cdot40+128\cdot38=13634,\\
 16D_0&=2\cdot205+28\cdot108+56\cdot24+224\cdot40=13738
\end{align*}
in the two cases.

\subsection{Evaluation of the two counting tables}
We evaluate the tables in hexacode coordinates.  Label the four positions of each tetrad by $\F_4$, and write
$\zeta^2+\zeta+1=0$.  Use the hexacode
\begin{equation}
 \mathcal H=\{(a,b,c,a+b+c,a+\zeta b+\zeta^2c,
                       a+\zeta^2b+\zeta c):a,b,c\in\F_4\}.
                                                               \label{eq:hexacode-for-count}
\end{equation}
The octads in the six-tetrad Golay description have exactly these forms:
\begin{description}[leftmargin=2.8em,style=nextline]
\item[Type A.] Two complete tetrads: $15$ choices.
\item[Type B.] A weight-four word $h\in\mathcal H$ with support $S$;
 in each $i\in S$ choose $\{0,h_i\}$ or its complement, with an even
 number of complements.  There are three words on each of the fifteen
 supports and eight choices of pairs, giving $360$ octads.
\item[Type C.] A word $h\in\mathcal H$ and a distinguished tetrad $j$;
 take $\F_4\setminus\{h_j\}$ there and $\{h_i\}$ in the other five.
 This gives $6\cdot64=384$ octads.
\end{description}
These are the usual hexacode coordinates of the Golay code
\cite{MacWilliamsSloane}.  Their counts also follow directly from
\eqref{eq:hexacode-for-count}: the hexacode has weights $0$, $4$, $6$, with
multiplicities $1$, $45$, $18$, and exactly three words on each four-set.
The three types total $759$.

The summand \eqref{eq:signed-octad-summand} depends only on the
following column intersections.  Put
\[
 g_i=|G\cap T_i|,\quad b_i=|H\cap T_i|,\quad
 h_i=|G\cap H\cap T_i|.
\]
Regard $D$, $E$, $F$ as their two-element sets of tetrad indices.  Test whether
\[
 \sum_{i\in D}g_i,\qquad \sum_{i\in E}b_i,\qquad\sum_i h_i
\]
are zero or four; these determine $d_1$, $d_2$, $d_0$.  Form
\[
 j_i=g_i+b_i-2h_i\quad(d_0=0),\qquad
 j_i=4-g_i-b_i+2h_i\quad(d_0=1).
\]
Test $\sum_{i\in F}j_i\in\{0,4\}$ to determine $d_3$.  Then
\begin{equation}
 d_4=\epsilon+d_0+d_1+d_2+d_3\pmod2,\qquad
 s=\tfrac12\sum_{i\in E}g_i+
                    \sum_{i\in D\cap E}b_i+\sum_{i\in D}h_i.
                                                               \label{eq:column-count-rule}
\end{equation}
Substitute these integers in \eqref{eq:signed-octad-summand}; if any
test fails, the contribution is zero.  Thus only the three column
vectors $(g_i)$, $(b_i)$, $(h_i)$ need to be counted.

A type-A source is
a union of columns.  A type-C source can be normalized to the word
$h=0$ by translating row labels by a hexacode word.  Such translations
preserve all three octad types.  For type B this follows from Hermitian
self-orthogonality of the hexacode: the number of pair complements
changed has parity
$\operatorname{Tr}_{\F_4/\F_2}\sum_i b_i h_i^2=0$.
A type-B source can similarly be normalized to the pairs $\{0,g_i\}$
on its support $S$.  Indeed the map
\[
 \mathcal H\longrightarrow\F_2^S,\qquad
 b\longmapsto\bigl(\operatorname{Tr}(b_i/g_i)\bigr)_{i\in S}
\]
has image the even four-bit code.  Its image is contained in that code
by self-orthogonality and is all of it because projection of the
hexacode to any three coordinates is a bijection.  Multiplying all row
labels by a common element of $\F_4^\times$ identifies the three nonzero
source words on each type-B support.

Write $\mathrm A_{ij}$ for the union of columns $i$, $j$,
$\mathrm B_S$ for a normalized type-B source on support $S$, and
$\mathrm C_j$ for a type-C source with zero word and distinguished column
$j$.  The admissible normalized source profiles in either setting are
\begin{equation}
\begin{aligned}
 &\mathrm A_{ij}\quad(1\le i<j\le6,\ \{i,j\}\ne\{1,2\}),\\
 &\mathrm B_{3456},\quad
 \mathrm B_{12ij}\quad(3\le i<j\le6),\quad
 \mathrm C_1,\ \mathrm C_2.
\end{aligned}
\label{eq:admissible-counting-profiles}
\end{equation}
Thus there are $14+7+2=23$ profiles in each setting.

For two normalized type-B words with supports $S$, $T$, put
$Z=\{i\in S\cap T:h_i/g_i=1\}$.  The possibilities, with their
multiplicities among the three nonzero hexacode words on $T$, are
\begin{equation}
\begin{array}{c|c|c}
 |S\cap T|&Z&\text{multiplicity}\\ \hline
 4&S\ (T=S)&1\\
 4&\varnothing&2\\
 3&\{i\},\ i\in S\cap T&1\text{ for each }i\\
 2&S\cap T&1\\
 2&\varnothing&2.
\end{array}                                                     \label{eq:hexacode-ratio-table}
\end{equation}
For intersection two, Hermitian orthogonality makes the two ratios
equal.  For intersection three, they are the three distinct nonzero
field elements: equality at two positions, together with the common
zero coordinate, would contradict minimum distance four.  This proves
the table.  If the target pair-complement bits are $\eta_i$ with even
sum, their intersection with the source pairs has size
\begin{equation}
 h_i=\begin{cases}
       2(1-\eta_i),&i\in Z,\\
       1,&i\in(S\cap T)\setminus Z,\\
       0,&i\notin S\cap T.
      \end{cases}                                              \label{eq:pair-intersection-rule}
\end{equation}
Thus the type-B column of either counting table is a sum over fifteen
supports and even four-bit words, using
\eqref{eq:hexacode-ratio-table} and
\eqref{eq:pair-intersection-rule}.

For a normalized type-B source and a type-C target, set
$\nu_i=\operatorname{Tr}(h_i/g_i)$ on $S$.
Every even four-bit word $\nu$ occurs for eight hexacode words.  The
intersection sizes are $1-\nu_i$ in an ordinary target column and
$1+\nu_i$ in its distinguished column, and are zero outside $S$.
For a normalized type-C source, a type-B target depends only on its
support and its even pair-complement bits.  A type-C target depends
only on its zero coordinates.  These zero sets have sizes six, two,
and zero, with multiplicities one, three for each prescribed pair,
and eighteen, respectively.  These statements follow immediately
from \eqref{eq:hexacode-for-count} and its minimum distance.
They complete the evaluation rules for every entry of
\eqref{eq:signed-sextet-table} and \eqref{eq:signed-trio-table}.

For example, consider the sextet row of type $(2,2,2,2,0,0)$.
Normalize the source to pairs $\{0,1\}$ in columns $1$, $2$, $3$, $4$;
$(1,1,1,1,0,0)$ belongs to the hexacode.  Formula
\eqref{eq:column-count-rule} gives the following contributions from
type-B targets; omitted supports contribute zero:
\begin{equation}
\begin{array}{c|l}
 \text{target support}&\text{contribution}\\ \hline
 1234&20z+2z^3\\
 1235,\ 1236&4z^{-1}\text{ for each support}\\
 1345,\ 1346&4z^{-1}+4z\text{ for each support}.
\end{array}                                                     \label{eq:worked-hexacode-row}
\end{equation}
To obtain each row, take the choices of $Z$ in
\eqref{eq:hexacode-ratio-table}, use
\eqref{eq:pair-intersection-rule}, and impose the one even-parity
condition on the four bits.  This gives
$16z^{-1}+28z+2z^3$ in total.
For type-C targets the only contributing pairs
$(j,\nu_1\nu_2\nu_3\nu_4)$ are
\[
 (1,1010),\ (1,1100),\ (3,0000),\ (3,0110),
\]
with weights $1$, $-1$, $-1$, $1$, respectively.  Each occurs eight times, giving
$16z^{-1}+16z$.  Type-A targets contribute $4z+2z^3$.
Adding the three types gives
$32z^{-1}+48z+4z^3$, as in
\eqref{eq:signed-sextet-table}.  In particular its signed sum is
$P_G'(1)=-32+48+12=28$.

The following table records representative three-type subtotals.
The rows of \eqref{eq:signed-sextet-table} and \eqref{eq:signed-trio-table}
record their sums.  In particular, the two trio sources
$\mathrm B_{1234}$ and $\mathrm B_{1256}$ have the same combined
polynomial but different target-type contributions.
\begin{center}
\footnotesize
\setlength{\tabcolsep}{3pt}
\begin{tabular}{l|lll}
 source & type A targets & type B targets & type C targets\\ \hline
 sextet $E$ & $4z+6z^3+3z^9$ & $72z$ & $64z$\\
 sextet $T_1T_4$ & $4z^{-3}+2z+6z^3$ & $24z$ & $0$\\
 sextet $T_3T_4$ & $2z^{-9}+4z^{-3}+4z^3+2z^9$ & $24z^3$ & $0$\\
 sextet $\mathrm B_{3456}$ & $0$ & $24z^3$ & $16z^3$\\
 sextet $\mathrm B_{1234}$ & $4z+2z^3$ & $16z^{-1}+28z+2z^3$ & $16z^{-1}+16z$\\
 sextet $\mathrm B_{1245}$ & $0$ & $16z^{-1}+16z+8z^3$ & $16z$\\
 sextet $\mathrm C_1$ & $z$ & $9z^{-1}+18z+3z^3$ & $9z^{-1}+19z+3z^3$\\ \hline
 trio $E$ & $4z+8z^3+z^9$ & $24z+48z^3$ & $0$\\
 trio $T_3T_5$ & $4z^{-3}+2z+4z^3+2z^9$ & $24z$ & $64z$\\
 trio $\mathrm B_{3456}$ & $4z+2z^3$ & $4z^{-3}+54z+2z^3+2z^9$ & $32z$\\
 trio $T_1T_3$ & $2z^{-3}+4z^{-1}+4z+2z^3$ & $24z$ & $0$\\
 trio $\mathrm B_{1234}$ & $4z+2z^3$ & $2z^{-3}+4z^{-1}+24z$ & $0$\\
 trio $\mathrm B_{1256}$ & $6z$ & $2z^{-3}+4z^{-1}+22z+2z^3$ & $0$\\
 trio $\mathrm B_{1235}$ & $0$ & $24z$ & $16z$\\
 trio $\mathrm C_1$ & $2z$ & $18z$ & $20z$
\end{tabular}
\end{center}
For the trio source $\mathrm B_{1256}$, the six contributing type-A
targets have column pairs $12$, $15$, $16$, $25$, $26$, $56$, each of weight $1$.
Its type-A subtotal is $6z$, whereas that for $\mathrm B_{1234}$
is $4z+2z^3$.  The combined polynomial is
$2z^{-3}+4z^{-1}+28z+2z^3$ for both sources.

For the sextet source $\mathrm C_1$, the contributing type-B targets
have supports containing columns $1$, $3$, or the support $2456$.
The first six supports contribute $9z^{-1}+18z$, and $2456$ contributes
$3z^3$.  For type-C targets distinguish $j=1$ and $j=3$.
When $j=1$, only weight-four hexacode words occur; their zero pair is
$23$ or a pair disjoint from $23$, giving
$9z^{-1}+9z+3z^3$.  When $j=3$, the zero word and the three zero pairs
in $\{4,5,6\}$ give $10z$.  This gives the last sextet row.

Applying the column, ratio, and trace rules to each of the $23$ profiles
in \eqref{eq:admissible-counting-profiles}, and summing the three target
types, gives \eqref{eq:signed-sextet-table} and \eqref{eq:signed-trio-table}.
Each type-B support carries $24$ octads; each type-C distinguished
column carries $64$.  In the trio case, profiles $(0,4,4)$ and $(4,4,0)$
each contain four type-A and $24$ type-B octads, giving $28$.
A profile $(4,2,2)$ contains four type-B supports and two type-C
columns, giving $4\cdot24+2\cdot64=224$.  Together with
\eqref{eq:signed-octad-summand}, these counts prove
Proposition~\ref{prop:CP-tensor-identities}.

\newcommand{\etalchar}[1]{$^{#1}$}


\begin{thebibliography}{CCN{\etalchar{+}}85}

\bibitem[AGS07]{AschbacherGuralnickSegev}
M.~Aschbacher, R.~M. Guralnick, and Y.~Segev.
\newblock Elementary abelian {$2$}-subgroups of {Sidki}-type in finite groups.
\newblock {\em Groups Geom. Dyn.}, 1(4):347--400, 2007.

\bibitem[Asc97]{AschbacherThree}
M.~Aschbacher.
\newblock {\em {3-Transposition Groups}}, volume 124 of {\em Cambridge Tracts
  in Mathematics}.
\newblock Cambridge University Press, 1997.

\bibitem[CCN{\etalchar{+}}85]{Atlas}
J.~H. Conway, R.~T. Curtis, S.~P. Norton, R.~A. Parker, and R.~A. Wilson.
\newblock {\em {Atlas of Finite Groups}}.
\newblock Oxford University Press, 1985.

\bibitem[Con85]{ConwayMonster}
J.~H. Conway.
\newblock A simple construction for the {Fischer--Griess} monster group.
\newblock {\em Invent. Math.}, 79:513--540, 1985.

\bibitem[CP]{ConwayParker}
J.~H. Conway and R.~A. Parker.
\newblock A remarkable {Moufang} loop, with an application to the {Fischer}
  group {$\Fi_{24}$}.
\newblock Cited as ``in press'' by Conway in 1985; no published version is
  known.

\bibitem[Fis69]{FischerWarwick}
B.~Fischer.
\newblock Finite groups generated by {$3$}-transpositions.
\newblock University of Warwick Lecture Notes, 1969 (unpublished).

\bibitem[Fis71]{FischerThree}
B.~Fischer.
\newblock Finite groups generated by {$3$}-transpositions. {I}.
\newblock {\em Invent. Math.}, 13:232--246, 1971.

\bibitem[Gri82]{GriessFriendly}
R.~L. Griess, Jr.
\newblock The friendly giant.
\newblock {\em Invent. Math.}, 69:1--102, 1982.

\bibitem[Gri86]{GriessCodeLoops}
R.~L. Griess, Jr.
\newblock Code loops.
\newblock {\em J. Algebra}, 100:224--234, 1986.

\bibitem[H{\"o}h26]{AtlasLean}
G.~H{\"o}hn.
\newblock A constructive {ATLAS} of finite simple groups in {Lean}.
\newblock Research note, 2026.
\newblock Project archive:
  \url{https://zenodo.org/records/22986720}.
\newblock Project repository:
  \url{https://github.com/Moonshine-in-Kansas/atlas}.

\bibitem[HS21]{HallShpectorov}
J.~I. Hall and S.~Shpectorov.
\newblock The spectra of finite {$3$}-transposition groups.
\newblock {\em Arab. J. Math.}, 10:611--638, 2021.

\bibitem[MS77]{MacWilliamsSloane}
F.~J. MacWilliams and N.~J.~A. Sloane.
\newblock {\em {The Theory of Error-Correcting Codes}}.
\newblock North-Holland, 1977.

\bibitem[Nor88]{NortonFi24}
S.~P. Norton.
\newblock On the group {$\Fi_{24}$}.
\newblock {\em Geom. Dedicata}, 25:483--501, 1988.

\bibitem[Sey24]{SeysenMonster}
M.~Seysen.
\newblock A fast implementation of the {Monster} group: the {Monster} has been
  tamed.
\newblock {\em J. Comput. Algebra}, 9:100012, 2024.

\bibitem[Smi79]{SmithWidths}
S.~D. Smith.
\newblock Large extraspecial subgroups of widths {$4$} and {$6$}.
\newblock {\em J. Algebra}, 58:251--281, 1979.

\bibitem[Tim78]{TimmesfeldExtraspecial}
F.~G. Timmesfeld.
\newblock Finite simple groups in which the generalized {Fitting} group of the
  centralizer of some involution is extraspecial.
\newblock {\em Ann. of Math. (2)}, 107:297--369, 1978.
\newblock Correction, \emph{ibid.} \textbf{109} (1979), 413--414.

\bibitem[VD85]{VirotteThesis}
M.-M. Virotte-Ducharme.
\newblock {\em Couples fisch{\'e}riens presque simples}.
\newblock Th{\`e}se, Universit{\'e} Paris~7, 1985.

\bibitem[VD87]{VirotteDucharme}
M.-M. Virotte-Ducharme.
\newblock Une construction du groupe de {Fischer} {$\Fi_{24}$}.
\newblock {\em M{\'e}m. Soc. Math. France (N.S.)}, (27):1--74, 1987.

\bibitem[Wil09]{WilsonBook}
R.~A. Wilson.
\newblock {\em {The Finite Simple Groups}}, volume 251 of {\em Graduate Texts
  in Mathematics}.
\newblock Springer, 2009.

\end{thebibliography}
\end{document}